\documentclass{amsart}

\newtheorem{theorem}{Theorem}[section]
\newtheorem{corollary}[theorem]{Corollary}
\newtheorem{lemma}[theorem]{Lemma}
\newtheorem{proposition}[theorem]{Proposition}
\theoremstyle{definition}
\newtheorem{definition}[theorem]{Definition}
\newtheorem{example}[theorem]{Example}

\newtheorem{remark}[theorem]{Remark}

\numberwithin{equation}{section}

\begin{document}

\title[Slice regular composition operators]{Slice regular composition operators}

\author[G. B. Ren]{Guangbin Ren}

\thanks{This work was supported by the NNSF  of China (11071230), RFDP (20123402110068).}
\author[X. P. Wang]{Xieping Wang}
\address{Guangbin Ren, Department of Mathematics, University of Science and
Technology of China, Hefei 230026, China}
\email{rengb$\symbol{64}$ustc.edu.cn}
\address{Xieping Wang, Department of Mathematics, University of Science and
Technology of China, Hefei 230026,
China}\email{pwx$\symbol{64}$mail.ustc.edu.cn}

\keywords{ Slice regular functions, Regular composition, Denjoy-Wolff type theorem, Littlewood subordination principle, Composition operators.}
\subjclass[2010]{30G35, 32A26}
\begin{abstract}
In the article  the class of slice regular functions is shown to be  closed under a new  regular composition. The new regular composition turns out to be globally  defined in contrast to the locally defined version by Vlacci. Its advantage over  Vlacci's  is demonstrated by its associated  theory  of composition operators and dynamical systems for slice regular functions.  Especially,  the corresponding  Littlewood subordination principle and the Denjoy-Wolff type theorem can be  established.

\keywords{Slice regular functions \and Regular composition \and Denjoy-Wolff type theorem \and Littlewood subordination principle \and Composition operators}
\subjclass[2010]{30G35, 32A26}

\end{abstract}
\maketitle

\section{Introduction}

It is known that the composition $f(\varphi(x))$ of any  two formal power series
$$f(x) =\sum_{j=0}^\infty x^j a_j, \qquad \varphi(x) =\sum_{k=0}^\infty x^k b_k$$
is a formal power series when coefficients $a_j$, $b_k$ are taken in a commutative field and $b_0=0$. However, if the constant term $b_0$ of the power series $\varphi$ is not 0, the existence of the composition $f(\varphi(x))$ has been an open problem for many years.
Only recently has received some partial answers \cite{CC,GK} and these results are also extended to the  noncommutative setting by Vlacci \cite{Vlacci}  for  slice regular functions.

In this article a new slice regular composition of any two
slice regular functions will be introduced based on slice regular product.
It has the  advantage   over  Vlacci's   that
the new composition can thus  lead to   the theory  of composition operators for slice regular functions.

The theory of slice regular functions is initiated  by  Gentili and  Struppa
 \cite{GS3,GS4}.
It has  elegant  applications in the functional calculus for noncommutative operators \cite{Co5}.
The detailed up-to-date theory appears in the monograph \cite{GS2}.
The theory of slice regular functions is associated with the non-elliptic  differential operator with nonconstant coefficients given by
 $$|\underline x|^2\frac{\partial }{\partial x_0}+\underline x
 \sum_{j=1}^3 x_j \frac{\partial }{\partial x_j},
$$
where $\underline x$ is the vector part of the quaternion $x=x_0+\underline x\in \mathbb H$, see \cite{CGCS} for more details.
%
%
%
%
%

 We now introduce  the \textit{slice regular composition operator} $C_{\varphi}$  on the slice regular Hardy space $H^2(\mathbb B)$.
 Let $\varphi:\mathbb B\rightarrow \mathbb B$ be a slice regular function on the open unit ball $\mathbb B$ of quaternions and we define
 \begin{eqnarray*}
C_{\varphi}(f)(q)=\sum\limits_{n=0}^{\infty} \varphi^{\ast n}(q)a_n,
\end{eqnarray*}
where the  power series expansion of $f\in H^2(\mathbb B)$ takes  the form $$ f(q)=\sum\limits_{n=0}^{\infty}q^n a_n,$$
and $\varphi^{\ast n}$ is the $n$th regular power of $\varphi$ with respect to $\ast$-product.

It is worth  remarking here that $C_{\varphi}f$ is well defined and slice regular in $\mathbb B$ globally (see Section 4 below for details) while the Vlacci's version of slice regular composition is, in general, only defined in $\frac13 \mathbb B$ other than the whole ball $\mathbb B$, due to the Bohr phenomenon, see Section 3 for details.

Our first main result is
the  Littlewood subordination principle for slice regular composition.
\begin{theorem} {\bf(Littlewood Subordination Principle)}\label{first main result}
Let $\varphi:\mathbb B\rightarrow \mathbb B$ be a slice regular function and $\varphi(0)=0$, then $C_{\varphi}$ is a bounded composition operator on $H^2(\mathbb B)$ with norm $||C_{\varphi}||=1$.
\end{theorem}
%

We  remark that, without the restriction $\varphi(0)=0$,  the exact norm of a composition
operator $C_{\varphi}$ is still unknown even in the classical holomorphic Hardy space $H^2(\mathbb D)$.
In the present article, we can provide the exact norm of a composition operator
on the slice regular Hardy space $H^2(\mathbb B)$ in the special case of which is induced by the slice  regular M\"obius transformations.
We refer to \cite{Cowen,BFHS,GP} for related questions concerning the norms of holomorphic composition operators.
\begin{theorem}
Let $\varphi$ be a regular M\"{o}bius transformation of $\mathbb B$, then
the composition operators $C_{\varphi}$ is  bounded on $H^2(\mathbb B)$ with norm
$$||C_{\varphi}||=\bigg(\frac{1+|\varphi(0)|}{1-|\varphi(0)|}\bigg)^{\frac 12}.$$
\end{theorem}

When $\varphi(\mathbb B)\subset\subset \mathbb B$, we know that the composition operator is compact on regular Hardy space $H^p(\mathbb B)$.
\begin{theorem}
Let $\varphi:\mathbb B\rightarrow \mathbb B$ be a regular function on $\mathbb B$ such that
$||\varphi||_{\infty}<1$. Then  $C_{\varphi}$ is a compact operator on $H^p(\mathbb B)$ with $1\leq p\leq \infty$.
\end{theorem}

The feature of our new slice regular composition lies at its analogous
dynamical behaviors of the iterates of slice regular self-mappings as in the holomorphic setting.
In particular, we have  the Denjoy-Wolff type theorem for slice regular mappings.
\begin{theorem} {\bf(Denjoy-Wolff)}
Let $Id\neq f\in Aut(\mathbb B)$ be not elliptic automorphism in which $f(\mathbb B_I)\subseteq \mathbb B_I $ for some $I\in \mathbb S $, then the sequence of  $\{f^{\odot n}\}$ uniformly converges on every compact subset of $\mathbb B$ to a boundary fixed point of $f$.
\end{theorem}

The paper consists of eight sections besides the introduction, and its outline is as follows. In Section 2, we set up the basic notations and give the preliminary results. Section 3 is devoted to recalling the Vlacci's regular composition for slice regular functions. In Section 4, we define our new regular composition and investigate its basic properties. In Section 5, we will investigate some results about the dynamical behaviors of the iterates of slice regular self-mappings $f$ of $\mathbb B$ under the extra assumption that $f$ preserves at least one slice. The assumption is satisfied for the slice regular automorphism group
and all the  elementary functions in the slice regular theory.
In Section 6, we establish the Littlewood subordination principle with respect to our new regular composition. The boundedness and compactness of composition operators are considered in Section 7 and Section 8, respectively. Finally, Section 9 comes the conclusion of the paper.

\section{Preliminaries}
We recall in this section some preliminary definitions and results on slice regular functions. Let $\mathbb H$ denote the noncommutative, associative, real algebra of quaternions with standard basis $\{1,\,i,\,j, \,k\}$,  subject to the multiplication rules
$$i^2=j^2=k^2=ijk=-1.$$
 Every element $q=x_0+x_1i+x_2j+x_3k$ in $\mathbb H$ is composed by the \textit{real} part ${\rm{Re}}\, (q)=x_0$ and the \textit{imaginary} part ${\rm{Im}}\, (q)=x_1i+x_2j+x_3k$. The \textit{conjugate} of $q\in \mathbb H$ is then $\bar{q}={\rm{Re}}\, (q)-{\rm{Im}}\, (q)$ and its \textit{modulus} is defined by $|q|^2=q\overline{q}=|{\rm{Re}}\, (q)|^2+|{\rm{Im}}\, (q)|^2$. We can therefore calculate the multiplicative inverse of each $q\neq0$ as $ q^{-1} =|q|^{-2}\overline{q}$.
 Every $q \in \mathbb H $ can be expressed as $q = x + yI$, where $x, y \in \mathbb R$ and
$$I=\frac{{\rm{Im}}\, (q)}{|{\rm{Im}}\, (q)|}$$
 if ${\rm{Im}}\, q\neq 0$, otherwise we take $I$ arbitrarily such that $I^2=-1$.
Then $I $ is an element of the unit 2-sphere of purely imaginary quaternions,
$$\mathbb S=\big\{q \in \mathbb H:q^2 =-1\big\}.$$

For every $I \in \mathbb S $,  we  denote by $\mathbb C_I$ the plane $ \mathbb R \oplus I\mathbb R $, isomorphic to $ \mathbb C$, and, if $\Omega \subseteq \mathbb H$, by $\Omega_I$ the intersection $ \Omega \cap \mathbb C_I $. Also, for $R>0$, we will denote the open ball centred at the origin with radius $R$ by
$$B(0,R)=\big\{q \in \mathbb H:|q|<R\big\}.$$

We can now recall the definition of slice regularity.
\begin{definition} \label{de:  regular} Let $\Omega$ be a domain in $\mathbb H$. A function $f :\Omega \rightarrow \mathbb H$ is called \emph{slice} \emph{regular} if, for all $ I \in \mathbb S$, its restriction $f_I$ to $\Omega_I$ is \emph{holomorphic}, i.e. it has continuous partial derivatives and satisfies
$$\bar{\partial}_I f(x+yI):=\frac{1}{2}\left(\frac{\partial}{\partial x}+I\frac{\partial}{\partial y}\right)f_I (x+yI)=0$$
for all $x+yI\in \Omega_I $.
 \end{definition}
As shown in \cite {Co4}, a class of domains, the so-called symmetric slice domains naturally qualify as  domains of definition  of slice regular functions.
\begin{definition} \label{de:  domain}
Let $\Omega$ be a domain in $\mathbb H $.  $\Omega$ is called a \textit{slice domain}  if $\Omega$ intersects the real axis and $\Omega_I$  is a domain
of $ \mathbb C_I $  for any $I \in \mathbb S $.

Moreover,  if  $x + yI \in \Omega$ implies $x + y\mathbb S \subseteq \Omega $ for any $x,y \in \mathbb R $ and $I\in \mathbb S$, then
 $\Omega$  is called a \textit{symmetric slice domain}.
\end{definition}

From now on, we will omit the term `slice' when referring to slice regular functions. A natural notion of derivative can be given for regular functions as follows (see \cite{GS3,GS4}).
\begin{definition} \label{de: derivative}
Let $\Omega$ be a slice domain in $\mathbb H $, and let $f :\Omega \rightarrow \mathbb H$  be a regular function. The \emph{slice derivative} of $f$ at $q=x+yI$
is defined by
$$\partial_I f(x+yI):=\frac{1}{2}\left(\frac{\partial}{\partial x}-I\frac{\partial}{\partial y}\right)f_I (x+yI).$$
 \end{definition}
Notice that the operators $\partial_I$ and $\bar{\partial}_I $ commute, and $\partial_I f=\frac{\partial f}{\partial x}$ for regular functions. Therefore, the slice derivative of a regular function is still regular so that we can iterate the differentiation to obtain the $n$-th
slice derivative
$$\partial^{n}_I f=\frac{\partial^{n} f}{\partial x^{n}},\quad\,\forall \,\, n\in \mathbb N. $$

In what follows, for the sake of simplicity, we will directly denote the $n$-th slice derivative $\partial^{n}_I f$ by $f^{(n)}$ for every $n\in \mathbb N$.

As stated in \cite{GS4}, a quaternionic power series $\sum\limits_{n=0}^{\infty}q^n a_n$ with $\{a_n\}_{n \in \mathbb N} \subset \mathbb H$ defines a regular function in its domain of convergence, which proves to be an open ball $B(0,R)$ with $R$ equal to the radius of convergence of the power series. The converse result is also true.
\begin{theorem}{\bf(Taylor Expansion)}\label{eq:Taylor}
A function f is regular on $B = B(0,R) $ if and only if $f$ has a power series expansion
$$f(q)=\sum\limits_{n=0}^{\infty}q^n a_n\quad with \quad a_n=\frac{f^{(n)}(0)}{n!}.$$
\end{theorem}
We can recover the values of a regular function on a symmetric slice domain from its values on a single slice $ \mathbb C_I $, due to the following representation formula (a special case of a result in \cite{Co4}), which was proven in \cite{CGS}.
\begin{theorem}{\bf(Representation Formula)}\label{eq:formula}
Let $ f $ be a regular function on a symmetric slice domain $\Omega \subseteq \mathbb H $ and let $I \in \mathbb S.$ Then for all $ x+yJ \in \Omega $ with $J \in \mathbb S$, the following equality holds
$$ f(x+yJ)=\frac{1}{2}\Big(f(x+yI)+f(x-yI)\Big)+\frac{1}{2}JI\Big(f(x-yI)-f(x+yI)\Big).$$
\end{theorem}
A fundamental result in the theory of regular functions is described by the splitting lemma (see \cite{GS4}), which relates slice regularity to classical holomorphy.
\begin{lemma}{\bf(Splitting Lemma)}\label{eq:Splitting}
Let $f$ be a regular function on a slice domain $\Omega \subseteq \mathbb H $. Then for any $I\in \mathbb S$ and any $J\in \mathbb S$ with $J\perp I$, there exist two holomorphic functions $F,G:\Omega_I\rightarrow \mathbb C_I$ such that for every $z=x+yI\in \Omega_I $, the following equality holds
$$f_I(z)=F(z)+G(z)J.$$
\end{lemma}

The pointwise product of two regular functions is not, in general, regular. To maintain the regularity, a new multiplication operation, the $\ast$-product, was introduced. On open balls centered at the origin, the $\ast$-product of two regular functions can be defined by means of their power series expansions in \cite{GS5}, mimicking the standard multiplication of polynomials in a skew field. However,  the generalization to the symmetric slice domains is based on the extension lemma (see \cite{Co4}).
\begin{lemma}{\bf(Extension Lemma)}\label{eq:Extension}
Let $\Omega$ be a symmetric slice domain and choose $I\in \mathbb S$. If $f_I:\Omega_I \rightarrow \mathbb H $ is holomorphic, then setting
$$ f(x+yJ)=\frac{1}{2}\Big(f_I(x+yI)+f_I(x-yI)\Big)+\frac{1}{2}JI \Big(f_I(x-yI)-f_I(x+yI)\Big)$$
extends $f_I$ to a regular function $f:\Omega \rightarrow \mathbb H $. Moreover, $f$ is the unique extension and it is denoted by {\rm{ext}}($f_I$).
\end{lemma}

In order to define the regular product of $f$ and $g$, which are regular functions on a symmetric slice domain $\Omega \subseteq \mathbb H$,
we apply the splitting lemma to write
 $$f_I(z)=F(z)+G(z)J, \qquad g_I(z)=H(z)+K(z)J,$$
where   $F,G,H,K:\Omega_I \rightarrow \mathbb C_I $ are  holomorphic functions
and  $I,J\in\mathbb S$ with $I\perp J$.

Let $f_I\ast g_I:\Omega_I \rightarrow \mathbb C_I $ be the holomorphic function defined by
$$f_I\ast g_I(z)=\Big(F(z)H(z)-G(z) \overline{K(\bar{z})}\Big)+\Big(F(z)K(z)+G(z)\overline{H(\bar{z})}\Big)J.$$
The regular extension of  $f_I\ast g_I(z)$ is defined to be the regular product (see \cite{Co4}).
\begin{definition} \label{de: Extension}
Let $f,g$ be two regular functions on a symmetric slice domain $\Omega \subseteq \mathbb H $. The \emph{regular product} (or \emph{$\ast$-product}) of $f$ and $g$ is the function defined by
$$f \ast g(q)={\rm{ext}} (f_I\ast g_I)(q)$$
regular on $\Omega$.
\end{definition}
Notice that the $\ast$-product is associative and is not, in general, commutative. It can be described in terms of the usual pointwise product (see \cite{Co4,GS5}):

\begin{proposition} \label{prop:RP}
Let $f$ and $g$ be two regular functions on a symmetric slice domain $\Omega \subseteq \mathbb H $. Then for all $q\in \Omega$,
$$f\ast g(q)=
\left\{
\begin{array}{lll}
f(q)g(f(q)^{-1}qf(q)) \qquad \,\,if \qquad f(q)\neq 0;
\\
\qquad  \qquad  0\qquad  \qquad \qquad if \qquad f(q)=0.
\end{array}
\right.
$$
\end{proposition}
\begin{corollary}
Let $f$ and $g$ be two regular functions on a symmetric slice domain $\Omega \subseteq \mathbb H $ and let $q\in \mathbb H$. Then $f\ast g(q)=0$ if and only if $f(q)=0$ or $f(q)\neq0 \,\,and\,\, g(f(q)^{-1}qf(q))=0$.
\end{corollary}

The theory of zero  set of regular functions depends heavily on its conjugation and symmetrization.

\begin{definition} \label{de: R-conjugate}
Let $f$ be a regular function on a symmetric slice domain $\Omega \subseteq \mathbb H $ and suppose that $$f_I(z)=F(z)+G(z)J$$
is the splitting of $f$,  where  $I, J\in\mathbb S$ and  $I\perp J$.

Consider the holomorphic function
$$f^c_I(z)=\overline{F(\bar z)}-G(z)J.$$
The \emph{regular conjugate} of $f$ is the function defined by
$$f^c(q)={\rm{ext}}(f^c_I)(q).$$
The \emph{symmetrization} of $f$ is the function defined by
$$f^s(q)=f\ast f^c(q)=f^c\ast f(q).$$
Both $f^c$ and $f^s$ are regular functions on $\Omega$.
\end{definition}
 The relation among the zero sets of $f$, $f^c$ and $f^s$ has been fully characterized (see \cite{GS5,GS6}):
\begin{proposition}\label{prop:zero set}
Let $f$ be a regular function on a symmetric slice domain $\Omega$. For all $x$, $y\in \mathbb R$ with $x+y\mathbb S \subseteq \Omega$, the zeros of the regular conjugate $f^c$ on $x+y\mathbb S$ are in one-to-one correspondence with those of $f$. Moreover, the symmetrization $f^s$ vanishes exactly on the sets $x+y\mathbb S$ on which $f$ has a zero.
\end{proposition}

Similarly, there is a close relation between the norms of $f$ and $f^c$,  when restricted to a sphere $x+y\mathbb S$ (see \cite{Sarfatti}):

\begin{proposition}\label{prop:conjugate}
Let $f$ be a regular function on the unit ball $\mathbb B=B(0,1)$. For any sphere of the form $x+y\mathbb S$ contained in $\mathbb B$, the following equalities hold true
$$\sup\limits_{I\in\mathbb S}\big|f(x+yI)\big|=\sup\limits_{I\in\mathbb S}\big|f^c(x+yI)\big|, \qquad \inf\limits_{I\in\mathbb S}\big|f(x+yI)\big|=\inf\limits_{I\in\mathbb S}\big|f^c(x+yI)\big|.$$
\end{proposition}
We now recall more results from  \cite {GS2}.
\begin{theorem}{\bf(Identity Principle)}\label{th:IP-theorem}
Let $f$ be a regular function on a slice domain $\Omega\subseteq \mathbb H$. Denote by $\mathcal{Z}_f$ the zero set of $f$, $$\mathcal{Z}_f=\big\{q\in \Omega:f(q)=0\big\}.$$
If there exists an $I\in \mathbb S$ such that $\Omega_I\cap \mathcal{Z}_f$ has an accumulation point in $\Omega_I$, then $f$ vanishes identically on $\Omega$.
\end{theorem}
\begin{theorem}{\bf(Maximum Modulus Principle)}\label{th:MMP-theorem}
Let $f$ be a regular function on a slice domain $\Omega \subseteq \mathbb H $. If $|f|$ has a relative maximum in $\Omega$, then $f$ is constant in $\Omega$.
\end{theorem}
\begin{theorem}{\bf(Schwarz Lemma)}\label{th:SL-theorem}
Let $f:\mathbb B\rightarrow\mathbb B$ be a regular function. If $f(0)=0$, then
$$|f(q)|\leq|q|$$
for all $q\in \mathbb B$ and $$|f'(0)|\leq1.$$
Both inequalities are strict {\rm{(}}except at $q=0${\rm{)}} unless $f(q)=qu$ for some $u\in \partial\mathbb B$.
\end{theorem}

The quaternionic counterpart of complex Hardy spaces was considered in \cite{C}.

\begin{definition} \label{de: Hardy}
Let $f$ be a regular function on $\mathbb B$ and let $0<p<\infty$. Set
$$||f||_p=\sup_{I\in \mathbb S}\lim\limits_{r\rightarrow 1^-}\bigg(\frac{1}{2\pi}\int_{-\pi}^{\pi}\big|f(re^{I\theta})\big|^{p}d\theta\bigg)^{\frac{1}{p}},$$
and$$||f||_{\infty}=\sup_{q\in \mathbb B}|f(q)|.$$
Then for any $0<p\leq\infty$, the slice regular Hardy space $H^p(\mathbb B)$ is defined as
$$H^p(\mathbb B)=\big\{f:\mathbb B \rightarrow \mathbb H \,|\, f\mbox{ is regular and}\,\, ||f||_p<\infty\big\}.$$
\end{definition}
Obviously, $H^p(\mathbb B)$ is a real vector space. Moreover, $H^p(\mathbb B)$ is not only a left $\mathbb H$-module, but also a right $\mathbb H$-module with respect to $+$ and $\ast$.
In analogy with the complex case, the space $H^2(\mathbb B)$ is  a Hilbert $\mathbb H$-module (see \cite{C}).
\begin{proposition} \label{prop:H^2-norm}
Let $f\in H^2(\mathbb B)$ and let $f(q)=\sum\limits_{n=0}^{\infty}q^n a_n $ be its power series expansion. Then the 2-norm of $f$,
$$||f||_2=\sup_{I\in \mathbb S}\lim\limits_{r\rightarrow 1^-}\bigg(\frac{1}{2\pi}\int_{-\pi}^{\pi}\big|f(re^{I\theta})\big|^{2}d\theta\bigg)^{\frac{1}{2}}$$
coincides with
$$\bigg(\sum\limits_{n=0}^{\infty}|a_n|^2\bigg)^{\frac{1}{2}}.$$
Moreover, for any $r\in[0,1)$ the integral$$\frac{1}{2\pi}\int_{-\pi}^{\pi}\big|f(re^{I\theta})\big|^{2}d\theta=\sum\limits_{n=0}^{\infty}|a_n|^2r^{2n}$$
does not depend on $I\in\mathbb S$.
\end{proposition}

\section{Vlacci's regular composition}
One point in the setting of quaternions is  that
the regularity does not keep under
the usual composition.
Indeed,
the only case in which the regularity of $f\circ\varphi$ is maintained for any regular function $f$ is when the regular function $\varphi$ is slice preserving. If instead $\varphi$ has no constraints, then $f$ has to be an affine function of the form $f(q)=a+qb$ for some $a,b\in \mathbb H$.

Vlacci \cite{Vlacci}  introduces  a kind of regular composition for regular functions,
which preserves regularity locally.

Let $f:B(0,R)\rightarrow \mathbb H$ be regular  with Taylor expansion $$f(q)=\sum\limits_{n=0}^{\infty}q^n a_n,  $$ and let $\varphi: B(0, r)\rightarrow B(0,R)$ be another regular function with $0<r\leq\infty$ and
$$\varphi(q)=\sum\limits_{n=0}^{\infty}q^n b_n.$$
Starting from  the Taylor expansion
$$f^{\circ}\varphi(q)=\sum\limits_{n=0}^{\infty}q^n \frac{(f^{\circ}\varphi)^{(n)}(0)}{n!} $$
and inspired by the classical \emph{Fa\`{a} di Bruno} formula in the commutative setting,
 Vlacci introduces the  following regular composition of $f$ with $\varphi$.

\begin{definition} \label{de: R-composition}  The Vlacci's  regular composition of $f$ with $\varphi$ is formally defined by
\begin{eqnarray}\label{Composition coefficents1}
f^{\circ}\varphi(q)=\sum\limits_{n=0}^{\infty}q^n c_n,
\end{eqnarray}
whose coefficients are given by
$$c_0=f(b_0)$$
and
\begin{eqnarray}\label{Composition coefficents}\qquad c_n=\frac1{n!}\sum\limits_{d=1}^{n}B_{n,d}(b_1,2!b_2,\cdots,n!b_n)f^{(d)}(b_0),\quad n\geq1.
\end{eqnarray}
Here $B_{n,d}(d=1,2,\cdots,n)$  are the homogeneous parts of the corresponding noncommutative Bell polynomial $B_n$, which coincides with the classical Bell polynomial when restricted to $\mathbb C^n$.
\end{definition}
Denote
$$C(n, n_2, \ldots, n_d)={n-1\choose n_2}{n_2-1\choose n_3}\cdots{n_{d-1}-1\choose n_d}$$
The explicit expression of $B_{n,d}$ is given by
$$B_{n,d}(q_1,\cdots,q_n)=\sum\limits_{n> n_2>\cdots>n_d\geq1}C(n, n_2, \ldots, n_d)q_{n_d}q_{n_{d-1}-n_d}\cdots q_{n_2-n_3}q_{n-n_2}$$
for $n\geq d\geq2$ and
$$B_{n-1,0}=1\qquad and \qquad B_{n,1}(q_1,q_2,\ldots,q_n)=q_n ,\quad n\geq1.$$
These  can be verified by  applying the recursion equation
$$B_{n+1}=\sum\limits_{k=0}^n {n\choose k}B_{n-k}q_{k+1}$$
for $n=0,1,2,\cdots$ together with the initial condition
 $$B_0=1.$$
For more generally noncommutative setting, not only limited to quaternions, we refer to   \cite[Theorem 1]{Schimming} and the references therein for more details.

With the Vlacci approach, one can introduce some new regular compositions.
For example, one can modify Vlacci's regular composition by  changing the order of the factors in the summands  of (\ref{Composition coefficents}).
We thus can introduce new regular composition of $f$ with $\varphi$, denoted by $f_{\circ}\varphi$,  via
\begin{eqnarray}\label{Composition coefficents2}
f_{\circ}\varphi(q)=\sum\limits_{n=0}^{\infty}q^n d_n,
\end{eqnarray}
whose coefficients are given by
$$d_0=f(b_0),\qquad d_n=\frac1{n!}\sum\limits_{d=1}^{n}f^{(d)}(b_0)B_{n,d}(b_1,2!b_2,\cdots,n!b_n),\quad n\geq1.$$

At first sight, it seems that the definition of $f^{\circ}\varphi$ coincides with that of $f_{\circ}\varphi$. Unfortunately, the answer is negative in general, as the following example shown.

\begin{example}
Let $f(q)=q$ and $ \varphi(q)=q^2I$ with $I,J\in \mathbb S$ and $I\perp J $, then
$$f^{\circ}\varphi(q)=q^2IJ,\qquad while \qquad f_{\circ}\varphi(q)=-q^2IJ.$$
\end{example}

Some new regular compositions can be further introduced.  Motivated by (\ref{eq:RC-con100}) below one can introduce  regular  compositions $f_\bullet \varphi$ and $f^\bullet \varphi$ via identities
\begin{eqnarray}\label{eq:RC-con100}
(f^{\bullet}\varphi)^c=({f^c})_{\circ}(\varphi^c),\qquad \quad(f_{\bullet}\varphi)^c=({f^{c}})^{\circ}(\varphi^c).
\end{eqnarray}
Here we use the fact that regular conjugations are involution operators.
\bigskip

There are strict restraints for the existence of the above slice regular compositions.

For any regular function $f(q)=\sum\limits_{n=0}^{\infty}q^n a_n$ on $B(0, R)$, denote
$$f_{abs}(q)=\sum\limits_{n=0}^{\infty}q^n |a_n|.$$
Vlacci \cite{Vlacci} provides  a sufficient condition for which  the regular compositions exist.

\begin{theorem} \label{Vlacci-existence}If the composition $f_{abs}\circ \varphi_{abs}$ exists on $B(0, R)$, then
$f^{\circ}\varphi$ is regular on $B(0, R)$.
\end{theorem}


However, the Bohr-type phenomenon \cite{RGS,GS2} shows that
$$\varphi_{abs}: \frac{1}{3}\mathbb B\longrightarrow \mathbb B$$
for any regular function $\varphi: \mathbb B\longrightarrow \mathbb B$.
Since the range of the function $\varphi_{abs}$ plays a role in the above composition,
we thus observed that there exists obstacle to define
$${\varphi_a}^{\circ}\varphi_a.$$
 Here $\varphi_a$ is the regular M\"obius transformation of $\mathbb B$ (see (\ref{eq:mobius-transform})).

\section{Slice regular compositions}
Vlacci's slice regular composition is defined only locally, we need to introduce a globally defined  slice regular composition in order to achieve the theory of regular composition.

\begin{definition} \label{de: R-composition} Let $\Omega$ be  a symmetric slice domain
in $\mathbb H$ and $0<R\leq\infty$. For any two regular functions $\varphi:\Omega\rightarrow B(0,R)$
and  $f:B(0,R)\rightarrow \mathbb H$  with power series expansion $$f(q)=\sum\limits_{n=0}^{\infty}q^n a_n, $$
 a regular composition of $f$ with $\varphi$ can be defined  by
\begin{eqnarray}\label{de:definition of RC}
f^{\odot}\varphi(q)=\sum\limits_{n=0}^{\infty} \varphi^{\ast n}(q)a_n.
\end{eqnarray}
Alternatively, one can also define  regular composition of $f$ with $\varphi$ as
$$f_{\odot}\varphi(q)=\sum\limits_{n=0}^{\infty} a_n \ast \varphi^{\ast n}(q).$$
\end{definition}
These two regular compositions are different  and conjugate in some sense, see Theorem \ref{eq:RC-prop1}. This fact reflects the diversity of the noncommutative setting as we mentioned in the introduction.

Some remarks are in order.
\begin{remark}\label{eq:R1}
The two regular compositions of $f$ with $\varphi$ are both well-defined. Indeed,
for any fixed $q_0=x+yI\in\Omega$, let $M$ denote the maximum modulus  of $\varphi$ on the 2-sphere $[q_0]=\big\{x+yJ:J\in\mathbb S\big\}$, i.e.
$$ M=\max\limits_{q\in[q_0]} {|\varphi(q)|}<R,$$ since $\varphi(\Omega)\subseteq B(0, R)$. From Proposition \ref{prop:RP} it is easy to prove  by induction that
$$\big|\varphi^{\ast n}(q)\big|\leq M^n$$
for any $q\in [q_0]$ and any $n\in \mathbb N$.
Again  Proposition \ref{prop:RP} implies that for all $q\in [q_0],$
$$\sum\limits_{n=0}^{\infty} \big| \varphi^{\ast n}(q)a_n\big|\leq\sum\limits_{n=0}^{\infty}|a_n|M^n< \infty,$$
and
$$\sum\limits_{n=0}^{\infty} \big|a_n \ast \varphi^{\ast n}(q)\big|\leq\sum\limits_{n=0}^{\infty}|a_n|M^n< \infty, $$
owing to the absolute and uniform convergence of the power series $\sum\limits_{n=0}^{\infty}q^n a_n $ on the closed ball $\overline{B(0,M)}\subset B(0,R)$.
\end{remark}

\begin{remark}
If $\varphi$ is a slice preserving function, i.e. $\varphi(\Omega_I)\subseteq \mathbb C_I$ for any $I\in\mathbb S$, then the two kinds of regular compositions $f^{\odot}\varphi$ and $f_{\odot}\varphi$ both coincide with the usual composition $f\circ\varphi$.
\end{remark}

\begin{remark}
At first sight, it seems that the definitions of $f^{\circ}\varphi$ and $f_{\circ}\varphi$  coincide with those of $f^{\odot}\varphi$ and $f_{\odot}\varphi$ respectively. Unfortunately, the answer is negative in general, as shown by the following example.
\end{remark}
\begin{example}
Let $f(q)=q^2$ and $ \varphi(q)=q^2I+qJ$ with $I,J\in \mathbb S$ and $I\perp J $. Then a straightforward calculation gives
$$f^{\odot}\varphi(q)=f_{\odot}\varphi(q)=q^2(qI+J)^{\ast 2}=-q^4-q^2,$$
and
$$f^{\circ}\varphi(q)=f_{\circ}\varphi(q)=-q^4-\frac23q^3IJ-q^2.$$

\end{example}

\begin{remark}
As mentioned in the end of the preceding section, for a general regular function $f$, the radii of convergence of power series in (\ref{Composition coefficents1}) are unknown
and it is quite possible that the definition domains of $f^{\circ}\varphi$ may be quite small than that of $\varphi$.

The advantages of new regular composition over
 Vlacci's is that $f_{\odot}\varphi$ has the same definition domain  with that of $\varphi$,  under the right hypotheses about the domain and range of $f$ and $\varphi$.
 Moreover, the definition domain of $\varphi$ is only required to be a symmetric slice domain other than a ball.
\end{remark}

\begin{remark}
The two  regular compositions $f^{\odot}\varphi$ and $f_{\odot}\varphi$ are different and both not, in general, associative, as the following example shown.
\end{remark}

\begin{example}
Let $f(q)=q^2$, $ g(q)=1+qI$ and $ \varphi(q)=qJ$ with $I,J\in \mathbb S$ and $I\perp J $. Then a straightforward calculation gives
$$(f^\odot g)^\odot\varphi(q)=q^2+2qJI+1\qquad and \qquad f^{\odot} \left(g^{\odot} \varphi \right)(q)=-q^2+2qJI+1,$$
while
$$(f_{\odot} g)_{\odot}\varphi(q)=q^2+2qIJ+1\qquad and \qquad f_{\odot} \left(g_{\odot} \varphi \right)(q)=-q^2+2qIJ+1,$$
which tell us  that
$$(f^\odot g)^\odot\varphi\neq  f^{\odot} \left(g^{\odot} \varphi \right)\qquad and \qquad (f_{\odot} g)_{\odot}\varphi\neq f_{\odot} \left(g_{\odot} \varphi \right).$$
Moreover, $$g^{\odot}\varphi(q)=1+qJI,\qquad while \qquad g_{\odot}\varphi(q)=1+qIJ,$$
which shows that $$g^{\odot}\varphi\neq g_{\odot}\varphi.$$
\end{example}

\begin{remark}
The quaternionic counterparts of the identity $$(fg)\circ\varphi=(f\circ\varphi)(g\circ\varphi)$$ do not always hold. Namely, in general,
$$(f\ast g)^{\odot}\varphi\neq(f^{\odot}\varphi)\ast(g^{\odot}\varphi),\qquad
(f\ast g)_{\odot}\varphi\neq(f_{\odot}\varphi)\ast(g_{\odot}\varphi),$$
as  shown by  the following example.
\end{remark}
\begin{example}
Let $f(q)=q^2I$, $ g(q)=1+qJ$, $ \varphi(q)=qIJ$ and $\psi(q)=q+I$ with $I,J\in \mathbb S$ and $I\perp J $. Then a straightforward calculation gives
$$(f\ast g)^{\odot}\varphi(q)=q^3-q^2I,\qquad (f\ast g)_{\odot}\psi(q)=q^3IJ+q^2(I+3J)-q(3IJ+2)-I-J,$$
while
$$(f^{\odot}\varphi)\ast(g^{\odot}\varphi)(q)=-q^3-q^2I,\quad (f_{\odot}\psi)\ast(g_{\odot}\psi)(q)=q^3IJ+q^2(I-J)+q(IJ-2)-I-J,$$
which shows that
$$(f\ast g)^{\odot}\varphi\neq(f^{\odot}\varphi)\ast(g^{\odot}\varphi),\qquad
(f\ast g)_{\odot}\psi\neq(f_{\odot}\psi)\ast(g_{\odot}\psi).$$
\end{example}

However, the following property holds.
\begin{proposition}\label{eq:RC-product}
Let $f,g$ be two regular functions on $B(0,R)$ and Let $\varphi:\Omega\rightarrow B(0,R)$ be regular on a symmetric slice domain. Then
$$(f\ast g)^{\odot}\varphi=
(f^{\odot}\varphi)\ast\big(g^{\odot}
((f^{\odot}\varphi)^{-\ast}\ast\varphi\ast(f^{\odot}\varphi))\big) \quad \mbox{on} \;\; \Omega\setminus \mathcal{Z}_{(f^{\odot}\varphi)^s},$$
and
$$(f\ast g)_{\odot}\varphi=\big(f_{\odot}
((g_{\odot}\varphi)\ast\varphi\ast(g_{\odot}\varphi)^{-\ast})\big)\ast(g_{\odot}\varphi) \quad \mbox{on} \;\; \Omega\setminus \mathcal{Z}_{(g_{\odot}\varphi)^s}.$$
\end{proposition}
\begin{proof}
Let $f,g:\mathbb B\rightarrow\mathbb H$ be as described with Taylor expansion of the form
$$ f(q)=\sum\limits_{n=0}^{\infty}q^n a_n,\qquad g(q)=\sum\limits_{n=0}^{\infty}q^n b_n$$
respectively. From the very definition,
\begin{equation*}
\begin{split}
(f\ast g)^{\odot}\varphi
&=\sum\limits_{n=0}^{\infty}\varphi^{\ast n}\bigg(\sum\limits_{k=0}^na_{n-k}b_k\bigg)
=\sum\limits_{k=0}^{\infty}\varphi^{\ast k}\ast\bigg(\sum\limits_{n=0}^{\infty}\varphi^{\ast n}a_n\bigg)b_k \\
&=\sum\limits_{k=0}^{\infty}\varphi^{\ast k}\ast(f^{\odot}\varphi)b_k
=(f^{\odot}\varphi)\ast\sum\limits_{k=0}^{\infty}
\Big((f^{\odot}\varphi)^{-\ast}\ast\varphi\ast(f^{\odot}\varphi)\Big)^{\ast k}b_k\\
&=(f^{\odot}\varphi)\ast\big(g^{\odot}
((f^{\odot}\varphi)^{-\ast}\ast\varphi\ast(f^{\odot}\varphi))\big)
\end{split}
\end{equation*}
The other one can also be proved similarly and an alternative way can be achieved immediately by applying the following proposition.
\end{proof}

In contrast to the local existence of Vlacci's composition as shown in
Theorem \ref{Vlacci-existence}, our slice regular compositions exists  globally.

\begin{proposition}\label{eq:RC-prop1} Let  $\Omega$ be a symmetric slice domain in $\mathbb H$ and  $0<R\leq\infty$.
For any two regular functions  $f:B(0,R)\rightarrow \mathbb H$ and $\varphi:\Omega\rightarrow B(0,R)$, we have
$f^{\odot}\varphi$ and $f_{\odot}\varphi$ exist and are regular on  $\Omega$ satisfying
\begin{eqnarray}\label{eq:RC-con100}
(f^{\odot}\varphi)^c={f^c}_{\odot}\varphi^c,\qquad \quad(f_{\odot}\varphi)^c=f^{c\,{\odot}}\varphi^c.
\end{eqnarray}
\end{proposition}
\begin{proof}
First, we prove that both $f^{c\,{\odot}}\varphi^c$ and ${f^c}_{\odot}\varphi^c$ are well-defined on $\Omega$. According to Remark \ref{eq:R1}, it suffices to prove that $\varphi^c(\Omega)\subseteq B(0,R)$. Indeed, suppose $\varphi^c(p)=a \in \mathbb H\backslash B(0,R)$ for some $p=x+Iy \in \Omega$. Then $p$ is a zero of the regular function $\varphi^c-a$. By Proposition \ref{prop:zero set}, there exists $\widetilde{p}\in x+y\mathbb S\subseteq \Omega$ such that $(\varphi^c-a)^c=\varphi-\overline{a}$ vanishes at $\widetilde{p}$. Hence, $\varphi(\Omega)$ includes $\overline{a}\in \mathbb H\backslash B(0,R)$, which is a contradiction with the hypothesis $\varphi(\Omega)\subseteq B(0,R)$.

Let $\sum\limits_{n=0}^{\infty}q^n a_n $ be power series expansion of $f$, i.e.
$ f(q)=\sum\limits_{n=0}^{\infty}q^n a_n$.
According to the definition of regular composition,
\begin{eqnarray}\label{eq:RC-con1}
(f^{\odot}\varphi)^c=\sum\limits_{n=0}^{\infty} (\varphi^{\ast n}a_n)^c=\sum\limits_{n=0}^{\infty} \bar{a}_n\ast(\varphi^c)^{\ast n}=f^c_{\odot}\varphi^c,
\end{eqnarray}
and
\begin{eqnarray}\label{eq:RC-con2}
(f_{\odot}\varphi)^c=\sum\limits_{n=0}^{\infty} (a_n\ast\varphi^{\ast n})^c=\sum\limits_{n=0}^{\infty}(\varphi^c)^{\ast n}\bar{a}_n=f^{c\,{\odot}}\varphi^c
\end{eqnarray}
as desired.
\end{proof}

 The  approach in the proof of Theorem \ref{eq:RC-prop1} also immediately works to extend  Proposition \ref{prop:conjugate} from the unit ball to symmetric slice domains.
\begin{proposition}\label{prop:conjugate1}
Let $f:\Omega_1\rightarrow\Omega_2$ be a regular function with $\Omega_i\subseteq \mathbb H (i=1,2)$ two symmetric slice domains. Then  the regular conjugate function
 $f^c:\Omega_1\rightarrow \mathbb H$ has the following properties.

$(\emph{i})$\  $f^c$
is regular satisfying $f^c(\Omega_1)\subset \Omega_2$.

$(\emph{ii})$\  $f^c$ is bijective if and only if $f$ is.

$(\emph{iii})$\ For any $x, y\in\mathbb R$ such that  $x+y\mathbb S\subset \Omega_1$, $$\sup\limits_{I\in\mathbb S}\big|f(x+yI)\big|=\sup\limits_{I\in\mathbb S}\big|f^c(x+yI)\big|, \qquad \inf\limits_{I\in\mathbb S}\big|f(x+yI)\big|=\inf\limits_{I\in\mathbb S}\big|f^c(x+yI)\big|$$

$(\emph{iv})$\ There holds the identity
$$\sup\limits_{q\in\Omega_1}\big|f(q)\big|=\sup\limits_{q\in\Omega_1}\big|f^c(q)\big|,\qquad \inf\limits_{q\in\Omega_1}\big|f(q)\big|=\inf\limits_{q\in\Omega_1}\big|f^c(q)\big|.$$
\end{proposition}

\section{The Denjoy-Wolff type theorem}
The dynamical behaviors of the iterates of regular self-mappings $f$ of $\mathbb B$
is considered in this section, under the extra assumption that $f$ preserves at least one slice. The limit of the iterates turns out to be the regular M\"{o}bius transformation of the unit ball.

In view of Theorem \ref{Denjoy-Wolff} below, it is useful to recall the following definition given in \cite{GS55} and there Gentili and Vlacci provide  a complete description of the fixed-point set for regular M\"{o}bius transformations of a quaternionic variable.

\begin{definition}
A regular M\"{o}bius transformation of $\mathbb B$ with only one fixed point in $\mathbb B$ is called \textit{elliptic}.

A regular M\"{o}bius transformation of $\mathbb B$ without fixed points in $\mathbb B$ is called:
\begin{itemize}
  \item \textit{parabolic} if it has only one fixed point on  the boundary of $\mathbb B$;
  \item \textit{hyperbolic} if it has at least two fixed points on the boundary of $\mathbb B$.
\end{itemize}
In particular, a hyperbolic regular M\"{o}bius transformation of $\mathbb B$ with a sphere of fixed points on the boundary of $\mathbb B$ is called \textit{spherical-hyperbolic}.
\end{definition}

It is crucial to show the existence of iterates of regular self-mappings with respect to our slice regular compositions.

\begin{theorem}
Let $f$ be a regular self-mapping of $\mathbb B$ such that $f(\mathbb B_I)\subseteq \mathbb B_I $ for some $I\in \mathbb S $, then $f^{\odot n}$ is well defined as a regular self-mapping of $\mathbb B$ and $$f^{\odot n}={\rm{ext}}(f_I)^n$$ for any $n\in \mathbb N$.
\end{theorem}
\begin{proof}
Let $f:\mathbb B\rightarrow\mathbb B$ be given  as described.

First, we prove that ${\rm{ext}}(f_I)^n$ is well defined as a regular self-mapping of $\mathbb B$ for any $n\in \mathbb N$. Indeed, since the restriction $f_I$ of $f$ to $\mathbb B_I$
is a holomorphic self-mapping of the open unit disc $\mathbb B_I\subset \mathbb C_I$,
then we can consider the iterates $(f_I)^n$ of $f_I$, where $(f_I)^1=f_I$ and $$(f_I)^{n+1}=f_I\circ(f_I)^n, \qquad n=1,2,\cdots.$$ Consequently, for each $n\in \mathbb N$, $(f_I)^n$ is a holomorphic self-mapping of $\mathbb B_I$, which naturally induces a regular self-mapping of $\mathbb B$, say ${\rm{ext}}(f_I)^n$, by using regular extension. The assertion that ${\rm{ext}}(f_I)^n(\mathbb B)\subseteq \mathbb B$ follows from a convex combination identity in Lemma 3.3 of \cite{RW}.

Next, we prove that  $f^{\odot n}$ is well defined  and $$f^{\odot n}={\rm{ext}}(f_I)^n$$ for any $n\in \mathbb N$. The reasons are as following. We can formally define in any order the $n$-th regular composition of $f$, on some small neighborhood $\mathcal{O}_n\subset \mathbb B$ of $0$, via (\ref{de:definition of RC}), and then it suffices to prove that the $n$th regular composition in any order can regularly extend to $\mathbb B$ and coincides with ${\rm{ext}}(f_I)^n$ there.

For simplicity, we only consider the $n$th regular composition $$\underbrace{f^{\odot}(\cdots(f^{\odot}(f^{\odot}f)))}_{n \,\,\mbox{copies}},$$ and from the very definition its restriction $$(f^{\odot}(\cdots(f^{\odot}(f^{\odot}f))))_I$$ to $(\mathcal{O}_n)_I$ coincides with $(f_I)^n$ there, but $(f_I)^n$ and its regular extension ${\rm{ext}}(f_I)^n$ are well defined on $\mathbb B_I$ and $\mathbb B$, respectively. Therefore, we can regularly extend $f^{\odot}(\cdots(f^{\odot}(f^{\odot}f)))$ to the whole ball $\mathbb B$ and $$f^{\odot}(\cdots(f^{\odot}(f^{\odot}f)))={\rm{ext}}(f_I)^n.$$

Similarly, we can use the same arguments as before to prove that the $n$th regular composition in any other order is well defined on $\mathbb B$ and coincides with ${\rm{ext}}(f_I)^n$ as well. Thus we can denote by $f^{\odot n}$ the well defined $n$th regular composition for any $n\in \mathbb N$.
\end{proof}

Now we consider the slice iterates.

\begin{proposition}\label{dynamic behavior}
Let $f$ be a regular self-mapping such that $f(\mathbb B_I)\subseteq \mathbb B_I $ for some $I\in \mathbb S $ and suppose $\{f^{\odot n}\}$ has a subsequence which converges to a nonconstant function. Then $f$ is a regular M\"{o}bius transformation of $\mathbb B$.
\end{proposition}
\begin{proof}
Let $f:\mathbb B\rightarrow\mathbb B$ be as described. Then the sequence $\{(f^{\odot n})_I\}=\{(f_I)^n\}$ satisfies the assumption given in Lemma 2.50 of \cite{Cowen}, thus $f_I\in Aut(\mathbb B_I)$ and its regular extension $f={\rm{ext}}(f_I)$ is a regular M\"{o}bius transformation of $\mathbb B$.
\end{proof}

Moreover, we have the following Denjoy-Wolff type theorem.
\begin{theorem}{\bf(Denjoy-Wolff Type Theorem)}\label{Denjoy-Wolff}
Let $Id\neq f\in Aut(\mathbb B)$ be not elliptic automorphism in which $f(\mathbb B_I)\subseteq \mathbb B_I $ for some $I\in \mathbb S $, then the sequence of  $\{f^{\odot n}\}$ uniformly converges on every compact subset of $\mathbb B$ to a boundary fixed point of $f$.
\end{theorem}

\begin{proof}
For any fixed compact subset $C\subset\mathbb B$, there exists $r\in(0,1)$ such that $C\subseteq \overline{B(0,r)} \subset \mathbb B$. By assumption, the sequence $\{(f^{\odot n})_I\}=\{(f_I)^n\}$ satisfies the assumption given in Theorem 2.51 of \cite{Cowen}, thus $\{(f_I)^n\}$ uniformly converges  on compact set $\overline{(B(0,r))_I}\subset\mathbb B_I$ to a boundary fixed point of $f$. By a convex combination identity from Lemma 3.3 of \cite{RW}, $$\max\limits_{q\in\overline{B(0,r)}}|f(q)-a|=\max\limits_{z\in\overline{(B(0,r))_I}}|f(z)-a|,$$
which gives the uniform convergence of $\{f^{\odot n}\}$ on $C$, since $C\subseteq \overline{B(0,r)}$.

Alternatively, the uniform convergence of $\{f^{\odot n}\}$ on $C\subset\mathbb B$ can follow from the representation formula.

\end{proof}

\section{The Littlewood subordination principle}
In this section we establish the Littlewood subordination principle for regular functions.

\begin{theorem}{\bf(Littlewood Subordination Principle)}\label{th:LST-theorem}
Let $\varphi:\mathbb B\rightarrow \mathbb B$ be a regular function and $\varphi(0)=0$, then $$||f^{\odot}\varphi||_2\leq||f||_2,\qquad ||f_{\odot}\varphi||_2\leq||f||_2$$ for any $f\in H^2(\mathbb B)$.
\end{theorem}

 To prove this result, we need  some basic lemmas.
\begin{lemma}\label{eq:H^2 modulus}
Let $f\in H^2(\mathbb B)$, then for all $q\in \mathbb B$, $$|f(q)|\leq\bigg(\dfrac1{1-|q|^2}\bigg)^{\frac{1}{2}}||f||_2.$$
\end{lemma}
\begin{proof}
Consider the power series expansion of $f$
$$f(q)=\sum\limits_{n=0}^{\infty}q^n \hat{f}(n).$$
By the Cauchy-Schwarz inequality, we have
$$|f(q)|\leq\sum\limits_{n=0}^{\infty}|q|^n |\hat{f}(n)|\leq \bigg(\sum\limits_{n=0}^{\infty}|\hat{f}(n)|^{2}\bigg)^{\frac{1}{2}}\bigg(\sum\limits_{n=0}^{\infty}|q|^{2n}\bigg)^{\frac{1}{2}} =\bigg(\dfrac1{1-|q|^2}\bigg)^{\frac{1}{2}}||f||_2$$
for all $q\in \mathbb B$.

\end{proof}
\begin{lemma}\label{eq:H^2 convergence}
Let $\{f_n\}_{n\in \mathbb N}$ be a convergent sequence in $H^2(\mathbb B)$. Then $\{f_n\}_{n\in \mathbb N}$ converges uniformly on every compact subset of $\mathbb B $.
\begin{proof}
Suppose that $f_n \longrightarrow f\in H^2(\mathbb B) $, i.e. $||f_n-f||_2\longrightarrow 0$ \, \,as \, \,$n\longrightarrow \infty$, then for any compact set $C\subset \mathbb B$, there exists $r\in(0,1)$ such that $C\subseteq \overline{B(0,r)} \subset \mathbb B$, by Theorem \ref{eq:H^2 modulus} and the Maximum Modulus Principle \ref{th:MMP-theorem} we have
$$\sup \limits_{q \in C} \big|f_n(q)-f(q)\big|\leq\sup \limits_{|q|\leq r} \big|f_n(q)-f(q)\big|\leq\bigg(\dfrac1{1-r^2}\bigg)^{\frac{1}{2}}||f||_2\longrightarrow 0\quad as \,\,n\longrightarrow \infty,$$
which implies that $\{f_n\}_{n\in \mathbb N}$ converges uniformly to $f$ on $C$.
\end{proof}
\end{lemma}

\begin{lemma}\label{eq:H^2 proposition}Let $f\in H^2(\mathbb B) $ and $\varphi \in H^{\infty}(\mathbb B)$, then $f\ast\varphi$, $\varphi \ast f \in H^2(\mathbb B)$. Moreover,
\begin{eqnarray*}
||f\ast\varphi||_2&\leq&||\varphi||_{\infty}||f||_2,
\\
||\varphi \ast f||_2&\leq&||\varphi||_{\infty}||f||_2.
\end{eqnarray*}
\begin{proof}
By Proposition \ref{prop:RP}, we have for every $q\in\mathbb B$,
$$\big|f\ast\varphi(q)\big|\leq\big|f(q)\big|||\varphi||_{\infty},$$
and hence for any $0<p<\infty$,
\begin{eqnarray}\label{eq:*con}
\begin{split}
||f\ast\varphi||_p&=\sup\limits_{I\in\mathbb S}\lim_{r\rightarrow 1^{-}}\bigg(\frac{1}{2\pi}\int_{-\pi}^{\pi}\big|f\ast\varphi(re^{I\theta})
\big|^{p}d\theta\bigg)^{\frac1p}
\\
&\leq ||\varphi||_{\infty}\sup\limits_{I\in\mathbb S}\lim_{r\rightarrow 1^{-}}\bigg(\frac{1}{2\pi}\int_{-\pi}^{\pi}\big|f(re^{I\theta})\big|^{p}d\theta\bigg)^{\frac1p}\\
&=||\varphi||_{\infty}||f||_p.
\end{split}
\end{eqnarray}
In particular, $$||f\ast\varphi||_2\leq||\varphi||_{\infty}||f||_2.$$
As for $f\ast\varphi$, according to Proposition \ref{prop:H^2-norm} and Proposition \ref{prop:conjugate1} that
$$||f^c||_2=||f||_2 \qquad\ and \qquad\ ||\varphi^c||_{\infty}=||\varphi||_{\infty},$$
Together with (\ref{eq:*con}) and the fact that $(\varphi\ast f)^c=f^c\ast\varphi^c$, these yield
$$||\varphi \ast f||_2=||(\varphi \ast f)^c||_2=||f^c \ast\varphi^c||_2\leq||\varphi^c||_{\infty}||f^c||_2=||\varphi||_{\infty}||f||_2,$$
which completes the proof.
\end{proof}
\end{lemma}

Incidentally, the approach to the proof of the preceding lemma also immediately give a very simple proof of Theorem 5.17 in \cite{ACS1}, which is of independent interest, and is initially proved for the left multiplier in the special case  that $p=2$ by using an approximation argument and Runge's theorem.
\begin{proposition}
Let $\varphi_a$ is a regular M\"obius transformation of $\mathbb B$. then the multiplier operators $$ \mathcal{M}_{\varphi_a}^l:f\mapsto\varphi_a\ast f$$
and $$\mathcal{M}_{\varphi_a}^r:f\mapsto f\ast\varphi_a$$ are isometries on $H^2(\mathbb B)$ and $H^{\infty}(\mathbb B)$, respectively. Furthermore, $\mathcal{M}_{\varphi_a}^r$ is an isometry on $H^p(\mathbb B)$ for $p\neq2$.
\end{proposition}
\begin{proof}
From \cite{Stoppato1}, we know every regular M\"obius transformation $\varphi_a$ of $\mathbb B$ is of the form
$$\varphi_a(q)=(1-q\bar{a})^{-\ast}\ast(a-q)u=:M_au,$$
where $u\in\mathbb \partial \mathbb B$. For any $f\in H^2(\mathbb B)$, it follows from Proposition \ref{prop:H^2-norm} that
$$||\mathcal{M}_{\varphi_a}^lf||_2=||(\mathcal{M}_{\varphi_a}^lf)_I||_2=||(M_{\varphi_a})_I(u\ast f)_I||_2=||(u\ast f)_I||_2=||u\ast f||_2=||f||_2$$
and $$||\mathcal{M}_{\varphi_a}^rf||_2=||(\mathcal{M}_{\varphi_a}^rf)^c||_2=||(\varphi_a)^c\ast f^c||_2=||f^c||_2=||f||_2.$$
The penultimate equation follows from the fact that $(\varphi_a)^c$ is also a regular M\"obius transformation of $\mathbb B$ and what we just proved on the isometry
property of left multiplier operators.

The fact that both $\mathcal{M}_a^l$ and $\mathcal{M}_a^r$ are isometries from $H^{\infty}(\mathbb B)$ into itself follows immediately from Proposition \ref{prop:RP} and (iv) in Proposition \ref{prop:conjugate1}.

It remains to prove that $\mathcal{M}_{\varphi_a}^r$ is an isometry on $H^p(\mathbb B)$ for $p\neq2$. By the proof of Lemma \ref{eq:H^2 proposition}, it suffices to prove that $$||\mathcal{M}_{\varphi_a}^rf||_p\geq||f||_p$$
for any $f\in H^p(\mathbb B)$ with $||f||_p>0$. For any fixed $\varepsilon \in (0,||f||_p)$, from the very definition it follows that there are some $I_0\in\mathbb S$ and  some $r_0\in(0,1)$ such that
$$\bigg(\frac{1}{2\pi}\int_{-\pi}^{\pi}\big|f(r_0e^{I_0\theta})\big|^{p}d\theta\bigg)^{\frac1p}
\geq ||f||_p-\varepsilon,$$
which together with Proposition \ref{prop:RP} yields that
$$||\mathcal{M}_{\varphi_a}^rf||_p\geq
\bigg(\frac{1}{2\pi}\int_{-\pi}^{\pi}\big|f\ast\varphi_a(re^{I_0\theta})
\big|^{p}d\theta\bigg)^{\frac1p}
\geq \big(||f||_p-\varepsilon \big)\min\limits_{|q|=r}|\varphi_a(q)|$$
for any $r\in (r_0,1)$. Therefore, letting $r\longrightarrow 1$ yields
$$||\mathcal{M}_{\varphi_a}^rf||_p\geq ||f||_p-\varepsilon,$$
which completes the proof.
\end{proof}
We are now in a position to prove the main theorem of this section.

\begin{proof}[Proof of Theorem $\ref{th:LST-theorem}$]
We define a translation operator $T:H^2(\mathbb B)\rightarrow H^2(\mathbb B)$ via
$$(Tf)(q)=\sum\limits_{n=0}^{\infty}q^n \hat{f}(n+1)$$
for any $f\in H^2(\mathbb B)$ with the Taylor expansion
$$f(q)=\sum\limits_{n=0}^{\infty}q^n \hat{f}(n)\in H^2(\mathbb B).$$
Due to Proposition \ref{prop:H^2-norm} ,
$$||Tf||_2^2=\sum\limits_{n=0}^{\infty}\big|\hat{f}(n+1)\big|^2=\sum\limits_{n=1}^{\infty}\big|\hat{f}(n)\big|^2\leq||f||^2_2$$
and hence $Tf\in H^2(\mathbb B)$.
Moreover, by definition,
\begin{eqnarray}\label{eq: Tf}
f(q)=f(0)+q(Tf)(q),\quad \forall \, q\in \mathbb B,
\end{eqnarray}
and
\begin{eqnarray}\label{eq: T^nf}
(T^nf)(0)=\hat{f}(n),\quad \forall \,n\in \mathbb N.
\end{eqnarray}
for any $f\in H^2(\mathbb B)$.

We suppose first that $f$ is a regular polynomial. In this case, Proposition \ref{prop:RP} shows that $f^{\odot}\varphi$ is bounded on $\mathbb B$ and so is in $H^2(\mathbb B)$. Now we estimate its norm $||f^{\odot}\varphi||_2$. By (\ref{eq: Tf}) we have
$$f^{\odot}\varphi(q)=f(0)+\varphi\ast \big((Tf)^{\odot}\varphi\big)(q),\quad \forall \, q\in \mathbb B.$$
By assumption $\varphi(0)=0$, which implies that the second term on the right-hand side of the identity above takes value 0 at $q=0$ and so $q$ is a common factor in its power series expansion. Consequently, $$\big|\big|f^{\odot}\varphi\big|\big|_2^2=\big|f(0)\big|^2+\big|\big|\varphi\ast\big((Tf)^{\odot}\varphi\big)\big|\big|^2_2.$$
By Lemma \ref{eq:H^2 proposition}, we have $$\big|\big|\varphi\ast\big((Tf)^{\odot}\varphi\big)\big|\big|_2\leq\big|\big|\varphi\big|\big|_{\infty}\big|\big|(Tf)^{\odot}\varphi\big|\big|_2\leq\big|\big|(Tf)^{\odot}\varphi\big|\big|_2,$$
and hence
\begin{eqnarray}\label{eq: f norm}
\big|\big|f^{\odot}\varphi\big|\big|_2^2\leq\big|f(0)\big|^2+\big|\big|(Tf)^{\odot}\varphi\big|\big|^2_2.
\end{eqnarray}
Now successively substituting $Tf$, $T^2f$, $\cdot \cdot\cdot$, $T^nf$ for $f$ in (\ref{eq: f norm}) yields
$$\big|\big|(Tf)^{\odot}\varphi\big|\big|_2^2\leq\big|(Tf)(0)\big|^2+\big|\big|(T^2f)^{\odot}\varphi\big|\big|^2_2,$$
$$\big|\big|(T^2f)^{\odot}\varphi\big|\big|_2^2\leq\big|(T^2f)(0)\big|^2+\big|\big|(T^3f)^{\odot}\varphi\big|\big|^2_2,$$
$$\vdots$$
$$\big|\big|(T^nf)^{\odot}\varphi\big|\big|_2^2\leq\big|(T^nf)(0)\big|^2+\big|\big|(T^{n+1}f)^{\odot}\varphi\big|\big|^2_2.$$
Putting all these inequalities together, we get
\begin{eqnarray}\label{eq: f 2-norm}
\big|\big|f^{\odot}\varphi\big|\big|_2^2\leq\sum\limits_{k=0}^n\big|(T^kf)(0)\big|^2+\big|\big|(T^{n+1}f)^{\odot}\varphi\big|\big|^2_2.
\end{eqnarray}
for any $n\in \mathbb N.$

Now recall that $f$ is a polynomial. If we choose $n$ be the degree of $f$, then $T^{n+1}f$ vanishes identically and so does $(T^{n+1}f)^{\odot}\varphi$. It follows from (\ref{eq: T^nf}) and (\ref{eq: f 2-norm}) that
\begin{eqnarray}\label{eq: f H^2-norm}
\big|\big|f^{\odot}\varphi\big|\big|_2^2\leq\sum\limits_{k=0}^n\big|(T^kf)(0)\big|^2=\sum\limits_{k=0}^n\big|\hat{f}(k)\big|^2=||f||^2_2,
\end{eqnarray}
which shows that $||f^{\odot}\varphi||_2\leq||f||_2$ holds at least for any regular polynomials.

Now we consider the general case that $f\in H^2(\mathbb B)$. Let $f_n$ be the $n$-th partial sum of the power series expansion of $f$, i.e. $$f_n(q)=\sum\limits_{k=0}^n q^k\hat{f}(k).$$
Then $f_n$ is a polynomial and $$||f_n-f||^2_2=\sum\limits_{k=n+1}^{\infty}|\hat{f}(k)|^2\longrightarrow 0\qquad as\quad n\longrightarrow 0.$$ By (\ref{eq: f H^2-norm}), for all $n\in\mathbb N$,
\begin{eqnarray}\label{eq: f_n H^2-norm}
||f_n^{\odot}\varphi||_2\leq||f_n||_2\leq||f||_2.
\end{eqnarray}
Notice that $|\varphi(q)|\leq|q|$ in virtue of Theorem \ref{th:SL-theorem}. Using the arguments similar to those used in Remark \ref{eq:R1} and in Lemma \ref{eq:H^2 modulus} we obtain that for any fixed $r\in (0,1)$,
$$\sup\limits_{|q|\leq r}\big|f_n^\odot\varphi(q)-f^\odot\varphi(q)\big|
\leq\bigg(\dfrac1{1-r^2}\bigg)^{\frac{1}{2}}||f||_2
\longrightarrow 0\quad as \,\,n\longrightarrow \infty,$$
which implies that $\big\{f_n^\odot\varphi\big\}_{n\in \mathbb N}$ converges uniformly to $f^\odot\varphi$ on every compact subset $U\subset \mathbb B$.
Consequently, for any $r\in (0,1)$ and $I\in \mathbb S$, we have
\begin{eqnarray*}
\bigg(\frac{1}{2\pi}\int_{-\pi}^{\pi}\big|f^\odot \varphi(re^{I\theta})\big|^{2}d\theta\bigg)^{\frac12}&=&\lim\limits_{n\longrightarrow \infty} \bigg(\frac{1}{2\pi}\int_{-\pi}^{\pi}\big|f_n^\odot \varphi(re^{I\theta})\big|^{2}d\theta\bigg)^{\frac12}\\
&\leq&\limsup\limits_{n\longrightarrow \infty}||f_n^{\odot}\varphi||_2\\
&\leq&\limsup\limits_{n\longrightarrow \infty}||f_n||_2\\
&\leq&||f||_2,
\end{eqnarray*}
which implies that
$f^{\odot}\varphi\in H^2(\mathbb B)$ and $||f^{\odot}\varphi||_2\leq||f||_2.$ In the second inequality we have used Proposition \ref{prop:H^2-norm}. The penultimate and last equations follow by (\ref{eq: f_n H^2-norm}).

As for $f_{\odot}\varphi$, we can prove the result in the very same manner, it suffices to notice that
$$f(q)=f(0)+q(Tf)(q)=f(0)+(Tf)(q)\ast q,\quad \forall \, q\in \mathbb B,$$
and hence
$$f_{\odot}\varphi(q)=f(0)+ \big((Tf)_{\odot}\varphi\big)\ast\varphi(q),\quad \forall \, q\in \mathbb B.$$
An alternative easier method is the conjugate method as in Lemma \ref{eq:H^2 proposition}, which is the following
$$||f_{\odot}\varphi||_2=||(f_{\odot}\varphi)^c||_2=||f^{c\,{\odot}}\varphi^c||_2\leq||f^c||_2=||f||_2.$$
In the second equation we have used equation (\ref{eq:RC-con2}).
Now the proof is completed.
\end{proof}

Let $\varphi:\mathbb B\rightarrow \mathbb B$ be a regular function and $\varphi(0)=0$, then $\varphi$ induces two \emph{composition operators} $C_{\varphi}$ and $D_{\varphi}$ on $H^2(\mathbb B)$ defined by
$$C_{\varphi}(f)=f^{\odot}\varphi, \qquad \qquad D_{\varphi}(f)=f_{\odot}\varphi.$$
Obviously, $C_{\varphi}$ is a right $\mathbb H$-linear operator while $D_{\varphi}$ a left $\mathbb H$-linear operator on \emph{Hilbert} $\mathbb H$-module $H^2(\mathbb B)$. Namely,
$$C_{\varphi}(f\ast\lambda+g\ast\mu)=C_{\varphi}(f)\ast\lambda+C_{\varphi}(g)\ast\mu $$
and
$$D_{\varphi}(\lambda\ast f+\mu\ast g)=\lambda\ast D_{\varphi}(f)+\mu\ast D_{\varphi}(g)$$
for all $\lambda$, $\mu\in \mathbb H$ and $f$, $g\in H^2(\mathbb B)$. Moreover, Theorem \ref{th:LST-theorem} shows that $C_{\varphi}$ and $D_{\varphi}$ are two bounded composition operators.
\begin{corollary}
Let $\varphi:\mathbb B\rightarrow \mathbb B$ be a regular function and $\varphi(0)=0$, then $C_{\varphi}$ and $D_{\varphi}$ are two bounded composition operators on $H^2(\mathbb B)$ with norms $||C_{\varphi}||=||D_{\varphi}||=1$.
\end{corollary}
\begin{proof}
On one hand, it follows from Theorem \ref{th:LST-theorem} that $||C_{\varphi}||\leq1$ and $||D_{\varphi}||\leq1$. On the other hand, $C_{\varphi}(1)=D_{\varphi}(1)=1$, which implies that $||C_{\varphi}||\geq1$ and $||D_{\varphi}||\geq1$. Consequently, $||C_{\varphi}||=||D_{\varphi}||=1$. This completes the proof of Theorem \ref{first main result}.
\end{proof}

\section{Boundedness of composition operators}
   The boundedness of the  \emph{slice composition operators} $C_{\varphi}$ and $D_{\varphi}$ on $H^p(\mathbb B)$
  are studied in this section.

  Let $\varphi:\mathbb B\rightarrow \mathbb B$ be a regular function. It
   induces  operators $C_{\varphi}$ and $D_{\varphi}$ via
$$C_{\varphi}(f)=f^{\odot}\varphi, \qquad \qquad D_{\varphi}(f)=f_{\odot}\varphi.$$
We shall show that in $H^2(\mathbb B)$ we have
$$||C_{\varphi}||=||D_{\varphi}||=\bigg(\dfrac{1+|\varphi(0)|}{1-|\varphi(0)|}\bigg)^{\frac 12}$$
for any  regular M\"{o}bius transformation $\varphi$ of $\mathbb B$.

We first consider the variant of Theorem \ref{th:LST-theorem}
without the restriction $\varphi(0)=0$.

\begin{theorem}\label{eq:lst}
Let $\varphi:\mathbb B\rightarrow \mathbb B$ be a regular function such that $\varphi(\mathbb B_I)\subseteq \mathbb B_I $ for some $I\in \mathbb S $, then
$$||f^{\odot}\varphi||_2\leq \bigg(\dfrac{1+|\varphi(0)|}{1-|\varphi(0)|}\bigg)^{\frac 12}||f||_2,\qquad ||f_{\odot}\varphi||_2\leq \bigg(\dfrac{1+|\varphi(0)|}{1-|\varphi(0)|}\bigg)^{\frac 12}||f||_2$$
for any $f\in H^2(\mathbb B)$.
\end{theorem}
\begin{proof}
Consider the power series expansion of $f$
$$f(q)=\sum\limits_{n=0}^{\infty}q^n \hat{f}(n).$$
Taking $J\in\mathbb S$ with $J\perp I$, we can decompose $\hat{f}(n)$ in the form
$$\hat{f}(n)=\hat{g}(n)+\hat{h}(n)J,$$
where $\hat{g}(n)$, $\hat{h}(n)\in \mathbb C_I$ for all $n\in \mathbb N.$
Consequently,$$f(q)=g(q)+h(q)J$$
and$$||f||_2^2=||g||_2^2+||h||_2^2,$$
where $g(q)=\sum\limits_{n=0}^{\infty}q^n \hat{g}(n)$ and $h(q)=\sum\limits_{n=0}^{\infty}q^n \hat{h}(n)$ are two regular functions on $\mathbb B$ such that each of them maps $\mathbb B_I$ into itself respectively. By assumption, $\varphi(\mathbb B_I)\subseteq \mathbb B_I $, so that
\begin{eqnarray*}
(f_{\odot}\varphi)_I(z)&=&\sum\limits_{n=0}^{\infty}\big(\hat{g}(n)+\hat{h}(n)J\big)\ast \varphi_I^ n(z)\\
&=&\sum\limits_{n=0}^{\infty}\varphi_I^{n}(z)\hat{g}(n)+\sum\limits_{n=0}^{\infty}\big(\overline{\varphi_I(\bar z)}\big)^{n}\hat{h}(n)J\\
&=&g_I\circ \varphi_I(z)+h_I\circ \overline{\varphi_I(\bar z)}J.
\end{eqnarray*}

In the penultimate equation we have used the fact that $J \ast \psi(z)=\overline{\psi(\bar z)}J$ for any holomorphic function $\psi$ in the variable $z\in \mathbb C_I $ and $I$, $J \in \mathbb S$ with $I\perp J$, which can be obtained from the definition of $\ast$-product.

Now it follows from the classical theory of composition operators on the Hardy space $ H^2$ (cf. Corollary 3.7 in \cite{Cowen}) and Proposition \ref{prop:H^2-norm} that
\begin{eqnarray*}
\big|\big|f_{\odot}\varphi\big|\big|_2^2=\big|\big|(f_{\odot}\varphi)_I\big|\big|_2^2&=&\big|\big|g_I\circ \varphi_I\big|\big|_2^2+\big|\big|h_I\circ \overline{\varphi_I(\bar \cdot)}\big|\big|_2^2\\
&\leq& \dfrac{1+|\varphi(0)|}{1-|\varphi(0)|}\Big(||g_I||_2^2+||h_I||_2^2\Big)\\
&=&\dfrac{1+|\varphi(0)|}{1-|\varphi(0)|}\Big(||g||_2^2+||h||_2^2\Big)\\
&=&\dfrac{1+|\varphi(0)|}{1-|\varphi(0)|}||f||_2^2.
\end{eqnarray*}

As for $f^{\odot}\varphi$, we can prove the result in the very same manner, it suffices to notice that $$(f^{\odot}\varphi)_I(z)=g_I\circ \varphi_I(z)+h_I\circ \varphi_I(z)J.$$

As in Theorem \ref{th:LST-theorem}, an alternative easier method is the conjugate method, which is the following
$$\big|\big|f^{\odot}\varphi\big|\big|_2=\big|\big|(f^{\odot}\varphi)^c\big|\big|_2=\big|\big|f^c_{\odot}\varphi^c\big|\big|_2\leq \bigg(\dfrac{1+|\varphi^c(0)|}{1-|\varphi^c(0)|}\bigg)^{\frac 12}||f^c||_2= \bigg(\dfrac{1+|\varphi(0)|}{1-|\varphi(0)|}\bigg)^{\frac 12}||f||_2.$$
In the second equation we have used equation (\ref{eq:RC-con1}).
Now the proof is completed.
\end{proof}
\begin{corollary} \label{eq:RC-norm}
Let $\varphi:\mathbb B\rightarrow \mathbb B$ be a regular function such that $\varphi(\mathbb B_I)\subseteq \mathbb B_I $ for some $I\in \mathbb S $, then $C_{\varphi}$ and $D_{\varphi}$ are two bounded composition operators on $H^2(\mathbb B)$. Moreover, $$\bigg(\dfrac1{1-|\varphi(0)|^2}\bigg)^{\frac 12}\leq||C_{\varphi}||\leq\bigg(\dfrac{1+|\varphi(0)|}{1-|\varphi(0)|}\bigg)^{\frac 12}$$ and $$\bigg(\dfrac1{1-|\varphi(0)|^2}\bigg)^{\frac 12}\leq||D_{\varphi}||\leq\bigg(\dfrac{1+|\varphi(0)|}{1-|\varphi(0)|}\bigg)^{\frac 12}.$$
\end{corollary}
\begin{proof}
The two upper bounds follow from Theorem \ref{eq:lst}. To prove the two lower bounds, consider the restriction $C_{\varphi_I}$ of $C_{\varphi}$ to $ H^2(\mathbb B_I)$ and its adjoint $C^{\ast}_{\varphi_I}$ acting on the evaluation kernels $K_w$ in $ H^2(\mathbb B_I)$, then by the classical $ H^2$ theory (see for instance Corollary 2.11 in \cite{Cowen}), we have
\begin{eqnarray}\label{eq:C}
C^{\ast}_{\varphi_I}(K_w)=K_{\varphi_I(w)}
\end{eqnarray}
and$$ K_w(z)=\frac1{1-\overline{w}z}\qquad with \qquad||K_w||_2=\bigg(\dfrac1{1-|w|^2}\bigg)^{\frac12}$$
for all $z$, $w\in \mathbb B_I$.
It is easy to prove that
\begin{eqnarray}\label{eq:C-norm}
||C_{\varphi}||=||C_{\varphi_I}||.
\end{eqnarray}
Now it follows from (\ref{eq:C}) and (\ref{eq:C-norm}) that
\begin{eqnarray}\label{eq:C-norm4}
||C_{\varphi}||=||C_{\varphi_I}||=||C^{\ast}_{\varphi_I}||\geq\frac{||K_{\varphi_I(w)}||_2}{||K_w||_2}=\bigg(\frac{1-|w|^2}{1-|\varphi_I(w)|^2}\bigg)^{\frac12}.
\end{eqnarray}
Taking $w=0$ yields that
$$||C_{\varphi}||\geq\bigg(\frac{1}{1-|\varphi(0)|^2}\bigg)^{\frac12}.$$
It follows from (\ref{eq:RC-con1}) that$$(D_{\varphi}(f))^c=C_{\varphi^c}(f^c),$$
from which we can easily prove that $$||D_{\varphi}||=||C_{\varphi^c}||.$$
Consequently,$$||D_{\varphi}||\geq\bigg(\frac{1}{1-|\varphi^c(0)|^2}\bigg)^{\frac12}=\bigg(\frac{1}{1-|\varphi(0)|^2}\bigg)^{\frac12},$$
which completes the proof.
\end{proof}
If $\varphi$ is a regular M\"{o}bius transformation $\varphi_a$ of $\mathbb B$ (see \cite{Stoppato1}), i.e. \begin{equation}\label{eq:mobius-transform}
\varphi_a(q)=(1-q\bar{a})^{-\ast}\ast(a-q)=
(a-q)\ast(1-q\bar{a})^{-\ast}=a-(1-|a|^2)\sum\limits_{n=1}^{\infty}q^n\bar{a}^{n-1} \end{equation}
 for some $a\in\mathbb B $, we can obtain more precise result, which is as follows.
\begin{theorem}
Let $\varphi$ be a regular M\"{o}bius transformation of $\mathbb B$, then
the composition operators $C_{\varphi}$ and $D_{\varphi}$ are bounded on $H^2(\mathbb B)$ with norms
$$||C_{\varphi}||=||D_{\varphi}||=\bigg(\dfrac{1+|\varphi(0)|}{1-|\varphi(0)|}\bigg)^{\frac 12}.$$
\end{theorem}
\begin{proof}
Let $\varphi(q)=(1-q\bar{a})^{-\ast}\ast(a-q)=(a-q)\ast(1-q\bar{a})^{-\ast}$ for some $a\in\mathbb B_I $, then $\varphi(\mathbb B_I)=\mathbb B_I$, it follows from Corollary \ref{eq:RC-norm} that
\begin{eqnarray}\label{eq:C1}
||C_{\varphi}(f)||_2\leq\bigg(\dfrac{1+|\varphi(0)|}{1-|\varphi(0)|}\bigg)^{\frac 12}||f||_2
\end{eqnarray}
and
\begin{eqnarray}\label{eq:C2}
||D_{\varphi}(f)||_2\leq\bigg(\dfrac{1+|\varphi(0)|}{1-|\varphi(0)|}\bigg)^{\frac 12}||f||_2,
\end{eqnarray}
which are the two upper inequalities. Now we show that each of them is sharp respectively.
An easy calculation gives
\begin{eqnarray}\label{eq:C3}
1-|\varphi_I(w)|^2=\frac{(1-|a|^2)(1-|w|^2)}{|1-w\bar{a}|^2},
\end{eqnarray}
Substituting (\ref{eq:C3}) into (\ref{eq:C-norm4}) yields
$$||C_{\varphi}||\geq\bigg(\frac{|1-w\bar{a}|^2}{1-|a|^2}\bigg)^{\frac12},\qquad\forall\,\, w\in\mathbb B_I.$$
Now if $a=|a|e^{I\theta}$, we take $w=-re^{I\theta}$ so that
$$||C_{\varphi}||\geq\lim\limits_{r\rightarrow 1^-}\bigg(\frac{|1-w\bar{a}|^2}{1-|a|^2}\bigg)^{\frac12}=\bigg(\frac{1+|a|}{1-|a|}\bigg)^{\frac12}=\bigg(\frac{1+|\varphi(0)|}{1-|\varphi(0)|}\bigg)^{\frac12},$$
which implies that inequality (\ref{eq:C1}) is sharp and $$||C_{\varphi}||=\bigg(\dfrac{1+|\varphi(0)|}{1-|\varphi(0)|}\bigg)^{\frac 12}.$$
It follows from (\ref{eq:RC-con1}) that$$(D_{\varphi}(f))^c=C_{\varphi^c}(f^c),$$
from which we can easily prove that $$||D_{\varphi}||=||C_{\varphi^c}||=\bigg(\dfrac{1+|\varphi^c(0)|}{1-|\varphi^c(0)|}\bigg)^{\frac 12}=\bigg(\dfrac{1+|\varphi(0)|}{1-|\varphi(0)|}\bigg)^{\frac 12}.$$
Notice that $\varphi(\mathbb B_I)\subseteq\mathbb B_I$, the very definition of regular composition and the identity principle \ref{th:IP-theorem} allow us to conclude that
$$f^{\odot}\varphi={\rm{ext}}(f_I\circ\varphi_I),$$
and hence
$$(f^{\odot}\varphi)^{\odot}\varphi={\rm{ext}}((f^{\odot}\varphi)_I\circ\varphi_I)
={\rm{ext}}((f_I\circ\varphi_I)\circ\varphi_I)={\rm{ext}}(f_I\circ(\varphi_I\circ\varphi_I))
={\rm{ext}}(f_I)=f.$$
Namely,$$C_{\varphi}(C_{\varphi}(f))=f.$$
Consequently, $$||f||_2=||C_{\varphi}(C_{\varphi}(f))||_2\leq||C_{\varphi}||||C_{\varphi}(f)||_2,$$
that is,
\begin{eqnarray}\label{eq:C-norm1}
||C_{\varphi}(f)||_2\geq\frac{1}{||C_{\varphi}||}||f||_2=\bigg(\dfrac{1-|\varphi(0)|}{1+|\varphi(0)|}\bigg)^{\frac 12}||f||_2.
\end{eqnarray}
Similarly, we can prove that $D_{\varphi}(D_{\varphi}(f))=f,$
which implies that
$$||f||_2\leq||D_{\varphi}||||D_{\varphi}(f)||_2,$$
and hence
\begin{eqnarray}\label{eq:C-norm2}||D_{\varphi}(f)||_2\geq\frac{1}{||D_{\varphi}||}||f||_2=\bigg(\dfrac{1-|\varphi(0)|}{1+|\varphi(0)|}\bigg)^{\frac 12}||f||_2.
\end{eqnarray}
Finally, Let $\left\{f_n\right\}$ be a sequence that exhibits the norm of $C_{\varphi}$, then the substitution $g_n=C_{\varphi}g_n$ shows that the inequality (\ref{eq:C-norm1}) is sharp. Similarly, we can prove that the inequality (\ref{eq:C-norm2}) is sharp also.

As a result, we have proved that
$$\bigg(\dfrac{1-|\varphi(0)|}{1+|\varphi(0)|}\bigg)^{\frac 12}||f||_2\leq||C_{\varphi}(f)||_2\leq\bigg(\dfrac{1+|\varphi(0)|}{1-|\varphi(0)|}\bigg)^{\frac 12}||f||_2$$
and
$$\bigg(\dfrac{1-|\varphi(0)|}{1+|\varphi(0)|}\bigg)^{\frac 12}||f||_2\leq||D_{\varphi}(f)||_2\leq\bigg(\dfrac{1+|\varphi(0)|}{1-|\varphi(0)|}\bigg)^{\frac 12}||f||_2.$$
Moreover, these inequalities are best possible.
\end{proof}

Now we consider the $H^p(\mathbb B)$ version  of Theorem \ref{th:LST-theorem}.

\begin{theorem}
Let $\varphi:\mathbb B\rightarrow \mathbb B$ be a regular function such that $\varphi(\mathbb B_I)\subseteq \mathbb B_I $ for some $I\in \mathbb S $, then $C_{\varphi}$ and $D_{\varphi}$ are two bounded composition operators on $H^p(\mathbb B)$
for any $1\leq p<\infty$. Moreover,
$$ ||C_{\varphi}||_p \simeq  ||D_{\varphi}||_p \simeq (1-|\varphi(0)|^2)^{-\frac 1p}.$$

\end{theorem}
\begin{proof}
First applying the classical theory of composition operator on the Hardy space $ H^p$ (cf. Corollary 3.7 in \cite{Cowen}) to $f_I\circ\varphi_I$ yields
\begin{eqnarray}\label{eq:p-norm}
\bigg(\dfrac{1}{1-|\varphi(0)|^2}\bigg)^{\frac 1p}||f_I||_p\leq\big|\big|f_I\circ\varphi_I\big|\big|_p\leq\bigg(\dfrac{1+|\varphi(0)|}{1-|\varphi(0)|}\bigg)^{\frac 1p}||f_I||_p
\end{eqnarray}
for any $f\in H^p(\mathbb B)$.
Taking into account the fundamental convexity equality $$(a+b)^p\leq2^{p-1}(a^p+b^p)$$ for any $a$, $b\in [0,\infty)$ and $p\geq1$, we can easily prove that
\begin{eqnarray}\label{eq:p-norm1}
||f_I||_p\leq||f||_p\leq2^{2-\frac1p}||f_I||_p
\end{eqnarray}
and
\begin{eqnarray}\label{eq:p-norm2}
||f^c||_p\leq2^{\frac1p}||f||_p.
\end{eqnarray}
On one hand, it follows from (\ref{eq:p-norm}), (\ref{eq:p-norm1}) and the fact that $\big(C_{\varphi}(f)\big)_I=f_I\circ\varphi_I$ that
$$\big|\big|C_{\varphi}(f)\big|\big|_p\leq
2^{2-\frac1p}\big|\big|f_I\circ\varphi_I\big|\big|_p
\leq2^{2-\frac1p}\bigg(\dfrac{1+|\varphi(0)|}{1-|\varphi(0)|}\bigg)^{\frac 1p}||f||_p,$$
which together with (\ref{eq:p-norm2}) give
\begin{eqnarray*}\label{eq:p-norm3}
\big|\big|D_{\varphi}(f)\big|\big|_p&\leq&
2^{\frac1p}\big|\big|\big(D_{\varphi}(f)\big)^c\big|\big|_p\\
&=&2^{\frac1p}\big|\big|C_{\varphi^c}(f^c)\big|\big|_p\\
&\leq&4\bigg(\dfrac{1+|\varphi(0)|}{1-|\varphi(0)|}\bigg)^{\frac 1p}||f^c||_p\\
&\leq&2^{2+\frac1p}\bigg(\dfrac{1+|\varphi(0)|}{1-|\varphi(0)|}\bigg)^{\frac 1p}||f||_p
\end{eqnarray*}
for any $f\in H^p(\mathbb B)$.
On the other hand, $$\big|\big|C_{\varphi}(f)\big|\big|_p=\big|\big|\big (C_{\varphi}(f)\big)_I\big|\big|_p=\big|\big|f_I\circ\varphi_I\big|\big|_p\geq2^{-(2-\frac1p)}\bigg(\dfrac{1}{1-|\varphi(0)|^2}\bigg)^{\frac 1p}||f||_p,$$
which implies that
\begin{eqnarray*}
||D_{\varphi}(f)||_p&\geq&2^{-\frac1p}\big|\big|\big(D_{\varphi}(f)\big)^c\big|\big|_p\\
&=&2^{-\frac1p}\big|\big|C_{\varphi^c}(f^c)\big|\big|_p\\
&\geq&\frac14\bigg(\dfrac{1}{1-|\varphi(0)|^2}\bigg)^{\frac 1p}||f^c||_p\\
&\geq& 2^{-(2+\frac1p)}\bigg(\dfrac{1}{1-|\varphi(0)|^2}\bigg)^{\frac 1p}||f||_p
\end{eqnarray*}
for any $f\in H^p(\mathbb B)$.

To be concluded, we have proved
$$2^{-(2-\frac1p)}\bigg(\dfrac1{1-|\varphi(0)|^2}\bigg)^{\frac 1p}\leq||C_{\varphi}||\leq 2^{2-\frac1p}\bigg(\dfrac{1+|\varphi(0)|}{1-|\varphi(0)|}\bigg)^{\frac 1p}$$ and $$2^{-(2+ \frac1p)}\bigg(\dfrac1{1-|\varphi(0)|^2}\bigg)^{\frac 1p}\leq||D_{\varphi}||\leq2^{2+ \frac1p}\bigg(\dfrac{1+|\varphi(0)|}{1-|\varphi(0)|}\bigg)^{\frac 1p}$$
as desired.
\end{proof}

We remark that the same arguments  as in the proof of the preceding theorem leads to the following result. More precisely,
let $\varphi$ is a regular M\"{o}bius transformation of $\mathbb B$, then $C_{\varphi}$ and $D_{\varphi}$ are two bounded composition operators on $H^p(\mathbb B)(1\leq p<\infty)$. Moreover,
$$2^{-(2-\frac1p)}\bigg(\dfrac{1-|\varphi(0)|}{1+|\varphi(0)|}\bigg)^{\frac 1p}||f||_p\leq||C_{\varphi}(f)||_p\leq2^{2-\frac1p}\bigg(\dfrac{1+|\varphi(0)|}{1-|\varphi(0)|}\bigg)^{\frac 1p}||f||_p$$
and
$$2^{-(2+\frac1p)}\bigg(\dfrac{1-|\varphi(0)|}{1+|\varphi(0)|}\bigg)^{\frac 1p}||f||_p\leq||D_{\varphi}(f)||_p\leq2^{2+\frac1p}\bigg(\dfrac{1+|\varphi(0)|}{1-|\varphi(0)|}\bigg)^{\frac 1p}||f||_p.$$

\section{Compactness of composition operators}
The compactness of composition operators $C_{\varphi}$ and $D_{\varphi}$ on $H^p(\mathbb B)$  is studied in this section.

 We need to specify precisely how much the inducing map $\varphi$ has to compress
the unit ball into itself in order to insure that the operators $C_{\varphi}$ and $D_{\varphi}$ compress bounded subsets of $H^p(\mathbb B)$ into relatively compact ones.
The most drastic way $\varphi$ can compress the unit ball is to take it to a point,
in which case the resulting composition operators $C_{\varphi}$ and $D_{\varphi}$ have one dimensional range (the space of constant functions), and are therefore compact. The next result shows that this compactness persists if we merely assume that $\varphi(\mathbb B)$ is relatively compact in $\mathbb B$.

\begin{theorem}\label{Compactness Theorem}
Let $\varphi:\mathbb B\rightarrow \mathbb B$ be a regular function on $\mathbb B$ such that
$||\varphi||_{\infty}<1$. Then both $C_{\varphi}$ and $D_{\varphi}$ are compact operators on $H^p(\mathbb B)$ for any $1\leq p\leq \infty$.
\end{theorem}
\begin{proof}
In the case that $p=2$, the result can be easily proved due to the speciality of $H^2(\mathbb B)$. For each positive integer $N$ define the operator $S_N:H^2(\mathbb B)\rightarrow H^2(\mathbb B)$ via
$$S_Nf(q)=\sum\limits_{n=0}^N \varphi^{\ast n}\hat{f}(n)$$
for any $f\in H^2(\mathbb B)$ with the Taylor expansion
$$f(q)=\sum\limits_{n=0}^{\infty}q^n \hat{f}(n)\in H^2(\mathbb B).$$
Thus $S_N$ maps $H^2(\mathbb B)$ onto the right $\mathbb H$-linear span of the first $n$ regular powers of $C_{\varphi}$. Obviously, $S_N$ is a bounded, finite
rank operator on $H^2(\mathbb B)$. We claim that $$||S_N-C_{\varphi}||_2\longrightarrow 0\quad as \:\:n\longrightarrow \infty.$$
This follows from the calculation below:
\begin{eqnarray}\label{Norm estimate}
\begin{split}
\big|\big|(S_N-C_{\varphi})f\big|\big|_2&=\Big|\Big|\sum\limits_{n=N+1}^{\infty} \varphi^{\ast n}\hat{f}(n)\Big|\Big|_2 \leq\sum\limits_{n=N+1}^{\infty}|\hat{f}(n)| \, ||\varphi^{\ast n}||_2\\
&\leq\sum\limits_{n=N+1}^{\infty}|\hat{f}(n)|\, ||\varphi||^n_{\infty}\leq \bigg(\sum\limits_{n=N+1}^{\infty}|\hat{f}(n)|^{2}\bigg)^{\frac{1}{2}}
\bigg(\sum\limits_{n=N+1}^{\infty}||\varphi||^{2n}_{\infty}\bigg)^{\frac{1}{2}}\\ &=\frac{||\varphi||^{N+1}_{\infty}}{(1-||\varphi||^2_{\infty})^{\frac{1}{2}}}||f||_2.
\end{split}
\end{eqnarray}
Thus
$$||S_N-C_{\varphi}||_2
\leq\frac{||\varphi||^{N+1}_{\infty}}{(1-||\varphi||^2_{\infty})^{\frac{1}{2}}}\longrightarrow 0\quad as \:\:n\longrightarrow \infty.$$
This exhibits $C_{\varphi}$ as an operator norm limit of finite rank operators, so it is compact on $H^2(\mathbb B)$. Similarly, it turns out that $D_{\varphi}$ is compact on $H^2(\mathbb B)$ as well.
\end{proof}
To prove Theorem \ref{Compactness Theorem} for $p\neq2$,   we need an $H^p(\mathbb B)$ version of Lemma \ref{eq:H^2 modulus}.
\begin{lemma}\label{eq:H^p modulus}
Let $f\in H^p(\mathbb B)$ with $1\leq p<\infty$, then for any $q\in \mathbb B$, $$|f(q)|\leq\sqrt 2\bigg(\dfrac1{1-|q|^2}\bigg)^{\frac{1}{p}}||f||_p.$$
\end{lemma}
\begin{proof}
For any fixed $I\in\mathbb S$, we apply the splitting lemma to write
 $$f_I(z)=F(z)+G(z)J,$$
where   $F,G:\mathbb B_I \rightarrow \mathbb C_I $ are two holomorphic functions
and  $J\in\mathbb S$ with $J\perp I$. Then by the very definition, $F,G\in H^p(\mathbb B_I)$. Moreover, $||F||_p\leq||f||_p$ and  $||G||_p\leq||f||_p$. Applying the classical result(cf. Corollary 2.14 in \cite{Cowen}) from complex analysis to $F$ and $G$ yields
$$|F(z)|\leq\bigg(\dfrac1{1-|z|^2}\bigg)^{\frac{1}{p}}||F||_p,
\qquad|G(z)|\leq\bigg(\dfrac1{1-|z|^2}\bigg)^{\frac{1}{p}}||G||_p.$$
Therefore,
$$|f(z)|^2=|F(z)|^2+|G(z)|^2
\leq\bigg(\dfrac1{1-|z|^2}\bigg)^{\frac{2}{p}}\big(||F||^2_p+||G||^2_p\big)
\leq2\bigg(\dfrac1{1-|z|^2}\bigg)^{\frac{2}{p}}||f||_p^2$$
as desired.
\end{proof}

This leads to an  $H^p(\mathbb B)$  version of  the Cauchy inequality.

\begin{lemma}
Let $f\in H^p(\mathbb B)$ with $1\leq p\leq\infty$, then
 $$|f^{(n)}(0)|\leq\sqrt 2 \,n!\, e^{\frac{1}{p}}\bigg(1+\frac{np}{2}\bigg)^{\frac{1}{p}}||f||_p \quad\mbox{for}\quad1\leq p<\infty,$$
 and
$$|f^{(n)}(0)|\leq n!\,||f||_p \quad\mbox{for}\quad p=\infty.$$
\end{lemma}
\begin{proof}
When $p=\infty$, the result is the Cauchy inequality.

 When $1\leq p<\infty$, for any $I\in\mathbb S$ and $r\in(0,1)$, it follows from the  Cauchy integral formula that
$$f^{(n)}(0)=\frac {n!}{2\pi I}\int_{\partial B_I(0,r)}\frac{dz}{z^{n+1}}f(z),$$
which together with Lemma \ref{eq:H^p modulus} implies that
$$|f^{(n)}(0)|\leq
\sqrt 2\,\frac {n!}{r^n}\bigg(\frac 1{1-r^2}\bigg)^{\frac 1p}||f||_p,\quad \forall\:r\in(0,1).$$
Taking $r=\bigg(\dfrac {np}{np+2}\bigg)^{\frac12}$ yields
$$|f^{(n)}(0)|\leq
\sqrt 2\,n!\bigg(1+\frac{np}{2}\bigg)^{\frac 1p}\bigg(1+\frac{2}{np}\bigg)^{\frac n2}||f||_p
\leq
\sqrt 2\,n!\,e^{\frac 1p}\bigg(1+\frac{np}{2}\bigg)^{\frac 1p}||f||_p$$
\end{proof}

Finally, we come to prove Theorem \ref{Compactness Theorem} for the remaining case for $p\neq2$.
\begin{proof} Assume that $p\neq2$.
For each positive integer $N$ define the operator $S_N:H^p(\mathbb B)\rightarrow H^p(\mathbb B)$ via
$$S_Nf(q)=\sum\limits_{n=0}^N \varphi^{\ast n}\frac{f^{(n)}(0)}{n!}$$
for any $f\in H^p(\mathbb B)$ with the Taylor expansion
$$f(q)=\sum\limits_{n=0}^{\infty}q^n\frac{f^{(n)}(0)}{n!}.$$
Thus $S_N$ is a bounded, finite
rank operator on $H^p(\mathbb B)$ and $||S_N||_p$ at most $$\sqrt 2\,e^{\frac 1p}\sum\limits_{n=0}^N\bigg(1+\frac{np}{2}\bigg)^{\frac 1p}||\varphi||^n_{\infty}$$ for $p<\infty$,  and $$\sum\limits_{n=0}^N||\varphi||^n_{\infty}
=\frac{1-||\varphi||^{N+1}_{\infty}}{1-||\varphi||_{\infty}}$$ for $p=\infty$.

To show $S_N$ is compact, it suffices to prove that $$||S_N-C_{\varphi}||_p\longrightarrow 0\quad as \:\:n\longrightarrow \infty,$$
which follows  from the inequalities $$||S_N-C_{\varphi}||_p\leq\sqrt 2\,e^{\frac 1p}\sum\limits_{n= N}^{\infty}\bigg(1+\frac{np}{2}\bigg)^{\frac 1p}||\varphi||^n_{\infty},\qquad p<\infty$$ and $$||S_N-C_{\varphi}||_{\infty}
\leq\frac{||\varphi||^{N+1}_{\infty}}{1-||\varphi||_{\infty}}.$$
\end{proof}

\section{Conclusions}
We have introduced two slice regular compositions, which leads to the theory of slice
composition operators. In particular, we  established the Denjoy-Wolff type theorem about the  dynamical behaviors of the iterates and the Littlewood subordination principle for slice regular functions.
The slice regular theory is believed to be  based upon two kind of operators, i.e.,   the slice regular product and  slice regular compositions.


\bigskip

\noindent {\bf Acknowledgment} {This work was supported by the NNSF  of China (11371337), RFDP (20123402110068).}

\bibliographystyle{amsplain}

\begin{thebibliography}{99}


\bibitem{ACS1} Alpay D., Colombo F., Sabadini I.: Pontryagin-de Branges-Rovnyak spaces of slice hyperholomorphic functions. J. Anal. Math. \textbf{121}, 87-125 (2013)
\bibitem{BFHS} Bourdon P. S., Fry E. E., Hammond C., Spofford C. H.: Norms of linear-fractional composition operators. Trans. Amer. Math. Soc. \textbf{356}(6), 2459--2480 (2004)
\bibitem{CC} Chaumat J., Chollet A.: On composite formal power series. Trans. Amer. Math. Soc. \textbf{353}(4), 1691--1703(2001)
\bibitem{Co4} Colombo F., Gentili G., Sabadini I., Struppa D. C.: Extension results for slice regular functions of a quaternionic variable. Adv. Math., \textbf{222}(5),  1793--1808 (2009)
\bibitem{CGS} Colombo F., Gentili G., Sabadini I.: A Cauchy kernel for slice regular functions. Ann. Global Anal. Geom. \textbf{37}(4) , 361--378 (2010)
\bibitem{CGCS} Colombo F., Gonz¨¢lez-Cervantes J. O., Sabadini I.: A nonconstant coefficients differential operator associated to slice monogenic functions. Trans. Amer. Math. Soc. \textbf{365}(1), 303--318 (2013)
\bibitem{Co5} Colombo F., Sabadini I., Struppa D. C.: Noncommutative functional calculus. Theory and Applications of Slice Hyperholomorphic Functions. Progress in Mathematics, vol. 289, Birkh\"{a}user/Springer, Basel, 2011.
\bibitem{Cowen} Cowen C. C., MacCluer B. D.: Composition Operators on Spaces of Analytic Functions. Studies in Advanced Mathematics, CRC Press, Boca Raton, 1995.
\bibitem{C} de Fabritiis C., Gentili G., Sarfatti G.: Quaternionic Hardy spaces. Preprint 2013, www.math.unifi.it/users/sarfatti/Hardy.
\bibitem{GP}  Gallardo-Guti\'{a}rrez E. A., Partington J. R.: Norms of composition operators on weighted Hardy spaces. Israel J. Math. \textbf{196}, 273--283(2013).
\bibitem{GK} Gan X., Knox N.: On composition of formal power series. Int. J. Math. Math. Sci. \textbf{30}(12), 761--770 (2002)
\bibitem{GS5} Gentili G., Stoppato C.: Zeros of regular functions and polynomials of a quaternionic variable. Mich. Math. J. \textbf{56}(3), 655--667 (2008)
\bibitem{GS6} Gentili G., Stoppato C.: The zero sets of slice regular functions and the open mapping theorem. in Hypercomplex Analysis and Applications, ed. by I. Sabadini, F. Sommen. Trends in Mathematics (Birkh\"{a}user, Basel, 2011),pp. 95--107.
\bibitem{GS2} Gentili G., Stoppato C., Struppa D. C.: Regular functions of a quaternionic variable. Springer Monographs in Mathematics, Springer, Berlin-Heidelberg, 2013.
\bibitem{GSSV} Gentili G., Stoppato C., Struppa D. C., Vlacci F.: Recent developments for regular functions of a hypercomplex variable, in Hypercomplex Analysis, ed. by I. Sabadini, M. Shapiro, F. Sommen. Trends in Mathematics (Birkh\"{a}user, Basel, 2009), 165--186.

\bibitem{GS3} Gentili G.,  Struppa D. C.: A new approach to Cullen-regular functions of a quaternionic variable. C. R. Math. Acad. Sci. Paris, \textbf{342}(10), 741--744 (2006)
\bibitem{GS4} Gentili G.,  Struppa D. C.:  A new theory of regular functions of a quaternionic variable. Adv. Math. \textbf{216}(1), 279--301 (2007)
\bibitem{GS50} Gentili G.,  Struppa D. C.: Regular functions on the space of Cayley numbers. Rocky Mt. J. Math. \textbf{40}(1), 225--241 (2010)
\bibitem{GS55}  Gentili G., Vlacci F.: On fixed points of regular M\"{o}bius transformations over quaternions. Complex analysis and dynamical systems IV. Part 1, 75--82, Contemp. Math., 553, Amer. Math. Soc. Providence, RI, 2011.
\bibitem{Ghiloni1}  Ghiloni R., Perotti A.: Slice regular functions on real alternative algebras. Adv. Math. \textbf{226}(2), 1662--1691 (2011)
\bibitem{RW} Ren G. B., Wang X. P.: The growth and distortion  theorems for slice regular functions. submitted.
\bibitem{RGS}  Rocchetta C. D., Gentili G., Sarfatti G.: The Bohr theorem for slice regular functions. Math. Nachr. \textbf{285}(17-18), 2093--2105 (2012)
\bibitem{Sarfatti} Sarfatti G.: Elements of function theory in the unit ball of quaternions. Ph.D. Thesis, Universit\`{a} di Firenze, 2013.
\bibitem{Schimming} Schimming R., Rida S. Z.: Non commutative Bell polynomials. Int. J. Algebra Comput. \textbf{6}(5), 635--644 (1996)
\bibitem{Stoppato1} Stoppato C.: Regular M\"{o}bius transformations of the space of quaternions. Ann. Global Anal. Geom. \textbf{39}(4), 387--401 (2010)
\bibitem{Vlacci} Vlacci F.: Regular composition for slice-regular functions of quaternionic variable. in Advances in Hypercomplex Analysis, ed. by G. Gentili, I. Sabadini, M. V. Shapiro, F. Sommen, D. C. Struppa, Springer INdAM Series, Springer, Milan, 2013, pp. 141--147.
\end{thebibliography}

\end{document}